\documentclass[a4paper,10pt]{amsart}
\usepackage[hidelinks]{hyperref}
\usepackage{tikz}
\usetikzlibrary{shapes.geometric}
\usepackage{enumitem}
\usepackage{datetime}
\usepackage{amsthm}
\usepackage{booktabs}
\usepackage{caption}
\usepackage{url}
\newdateformat{monthyeardate}{%
	\monthname[\THEMONTH] \THEDAY, \THEYEAR}
\DeclareMathOperator{\myvec}{vec}
\DeclareMathOperator{\myrank}{rank}
\DeclareMathOperator{\mydist}{dist}
\DeclareMathOperator{\mymod}{mod}
\newtheorem{theorem}{Theorem}[section]
\newtheorem{lemma}[theorem]{Lemma}
\newtheorem{proposition}[theorem]{Proposition}
\newtheorem{corollary}[theorem]{Corollary}
\theoremstyle{definition}
\newtheorem{definition}[theorem]{Definition}
\newtheorem{remark}[theorem]{Remark}
\newtheorem{example}[theorem]{Example}
\newtheorem{problem}[theorem]{Problem}
\newtheorem{roadmap}[theorem]{Roadmap}
\begin{document}
\title{Biangular lines revisited}
\author{Mikhail Ganzhinov and Ferenc Sz\"oll\H{o}si}
%\dedicatory{\monthyeardate\today{}}
\address{M.\ G.: Department of Communications and Networking, Aalto University School of Electrical Engineering, P.O. Box 15400, 00076 Aalto, Finland.}
\email{mikhail.ganzhinov@aalto.fi}
\address{F.\ Sz.: Hungary/Japan/Finland, (formerly with Aalto University).}
\email{szoferi@gmail.com} 
\thanks{This research was supported in part by the Academy of Finland, Grant \#289002}
\begin{abstract}
Line  systems passing through the origin of the $d$ dimensional Euclidean space admitting exactly two distinct angles are called biangular. It is shown that the maximum cardinality of biangular lines is at least $2(d-1)(d-2)$, and this result is sharp for $d\in\{4,5,6\}$. Connections to binary codes, few-distance sets, and association schemes are explored, along with their multiangular generalization.
\end{abstract}
\maketitle
\section{Introduction}
This paper is concerned with optimal arrangements of unit vectors in Euclidean space. Let $d,m,s\geq 1$ be integers, let $\mathbb{R}^d$ denote the $d$-dimensional Euclidean space with standard inner product $\left\langle .,.\right\rangle$, and let $\mathcal{X}\subset\mathbb{R}^d$ be a set of unit vectors with the associated set of inner products $A(\mathcal{X}):=\{\left\langle x,x'\right\rangle\colon x\neq x';x,x'\in\mathcal{X}\}$. The following two concepts are central to this paper: $\mathcal{X}$ forms a spherical $s$-distance set \cite{BBS}, \cite{L}, \cite{NM}, \cite{SHINO} if $|A(\mathcal{X})|\leq s$; and $\mathcal{X}$ spans a system of $m$-angular lines (passing through the origin in the direction of $x\in\mathcal{X}$), if $-1\not\in A(\mathcal{X})$ and $|\{\gamma^2\colon \gamma\in A(\mathcal{X})\}|\leq m$. With this terminology a system of $m$-angular lines can be considered as the switching class of certain spherical $2m$-distance set without antipodal vectors. If the parameters $s$ and $m$ are not specified, then we talk about few-distance sets \cite{BLO}, and multiangular lines, respectively. The fundamental question of interest concerns the maximum cardinality and structure of the largest sets $\mathcal{X}$ and their corresponding $A(\mathcal{X})$.

Equiangular lines (i.e., the case $m=1$) are classical combinatorial objects \cite{DGS}, \cite{LS}, \cite{vLS}, receiving considerable recent attention, see e.g., \cite{SUD}, \cite{GKM}. Biangular lines correspond to the case $m=2$, which have also been the subject of several recent studies \cite{Best}, \cite{KHA}, \cite{BTF}, \cite{HOLZNEW}, \cite{NS} where in particular engineers investigated them focusing on tight frames \cite{STEINER}, \cite{SHAYNE}. Our motivation for studying these objects is fueled by their intrinsic connection to kissing arrangements \cite{COHN}, \cite{SPLAG}, \cite{DUSTIN}. In particular, we hope that the techniques and results described in this paper will eventually contribute to a deeper understanding of low-dimensional sphere packings. The goal of this paper, which heavily builds on the theory set forth earlier in \cite{SZO}, is to describe a systematic approach to the study of multiangular lines, focusing in particular on the biangular case.

The outline of this paper is as follows: in Section~\ref{sec2} we give various constructions of biangular lines, showing that their maximum number is at least $2(d-1)(d-2)$ in $\mathbb{R}^d$ for every $d\geq 3$. In Section~\ref{sec3} we set up a general computational framework for exhaustively generating all (sufficiently large) biangular line systems, and in Section~\ref{sec4} we leverage on these ideas to classify the largest sets in $\mathbb{R}^d$ for every $d\leq 6$. In Section~\ref{sec5} we present our results on multiangular lines. In Section~\ref{sec6} we conclude our manuscript with a selection of open problems. To improve the readability, a technical part on graph representation was moved to Appendix~\ref{APPENDIX1}, along with a few rather large matrices displayed in Appendix~\ref{APPENDIXB}.

For a convenient reference, we display here in Table~\ref{TABLEMAIN} the best known lower bounds on the maximum number of biangular lines in $\mathbb{R}^d$ (where entries marked by $\ast$ are exact). Each of these numbers are new, except for the well-known case in dimension $23$.
\begin{table}[htbp]
\centering
\caption{Lower bounds on the maximum number of biangular lines in $\mathbb{R}^d$}\label{TABLEMAIN}
\begin{tabular}{ccccccccccc}
\toprule
$d$  & 2 & 3  &  4 &  5 &  6 &  7 &   8 &   9 &  10 & 11\\
$\#$ & 5$^\ast$ & 10$^\ast$ & 12$^\ast$ & 24$^\ast$ & 40$^\ast$ & 72 & 126 & 240 & 256 & 276\\
\midrule
$d$  & 12 & 13 & 14 & 15  &  16 &  17-20 & 21 & 22 &  23-35 & 36-\\
$\#$ & 296 & 336 & 392 & 456 & 576 & 816 & 896 & 1408 & 2300 & $2(d-1)(d-2)$\\
\bottomrule
\end{tabular}
\end{table}
%%%%%%%%%%%%%%%%%%%%%%%%%%%%%%%%%%%%%%%%%%%%%%%
% SECTION 2: CONSTRUCTIONS OF BIANGULAR LINES %
%%%%%%%%%%%%%%%%%%%%%%%%%%%%%%%%%%%%%%%%%%%%%%%
\section{Constructions of biangular lines}\label{sec2}
The goal of this section is to give various explicit constructions of large biangular line systems in low dimensional spaces.

Let $\mathcal{X}\subset\mathbb{R}^d$ be a set of unit vectors, spanning biangular lines and let $O$ be an orthogonal matrix representing an isometry of $\mathbb{R}^d$. Since for every $x\in\mathcal{X}$ the sets $\mathcal{X}':=(\mathcal{X}\setminus\{x\})\cup\{-x\}$ and $\mathcal{X}'':=\{Ox\colon x\in\mathcal{X}\}$ span the same system of biangular lines as $\mathcal{X}$, we may replace any $x\in\mathcal{X}$ with its negative or apply $O$ whenever it is necessary. Throughout this section we represent biangular line systems with a (conveniently chosen) corresponding set of unit vectors, and uniqueness is understood up to these operations.

First we give an elementary proof to the following trivial warm-up result.
\begin{lemma}\label{TENGONLEMMA}
The $5$ lines passing through the antipodal vertices of the convex regular $10$-gon is the unique maximum biangular line system in $\mathbb{R}^2$.
\end{lemma}
\begin{proof}
Let $n\geq 1$, let $\alpha,\beta\in\mathbb{R}$ such that $0\leq \alpha < \beta<1$, and assume that $\mathcal{X}:=\{x_i\colon i\in\{1,\dots,n\}\}$ spans a maximum biangular line system in $\mathbb{R}^2$ with corresponding set of inner products $A(\mathcal{X})\subseteq\{\pm\alpha, \pm \beta\}$. We may assume without loss of generality that $x_1=[1,0]$. Since for $i\in\{2,\dots,n\}$ we have $\left\langle x_1,x_i\right\rangle\in A(\mathcal{X})$, it immediately follows that $x_i\in\{[\alpha,\sqrt{1-\alpha^2}],[\alpha,-\sqrt{1-\alpha^2}],[\beta,\sqrt{1-\beta^2}],[\beta,-\sqrt{1-\beta^2}]\}$, after replacing $x_i$ by $-x_i$ if it is necessary. Therefore $n\leq 5$, and the claimed configuration is indeed a largest possible example.

To see uniqueness, let us use the notation $x_2=[\alpha,\sqrt{1-\alpha^2}]$, $x_3=[\alpha,-\sqrt{1-\alpha^2}]$, $x_4=[\beta,\sqrt{1-\beta^2}]$, and $x_5=[\beta,-\sqrt{1-\beta^2}]$. Since $\left\langle x_2,x_3\right\rangle=2\alpha^2-1$, $\left\langle x_4,x_5\right\rangle=2\beta^2-1$, and $\left\langle x_2,x_4\right\rangle+\left\langle x_2,x_5\right\rangle=2\alpha\beta$, the following system of polynomial equations in the variables $\alpha$ and $\beta$ must hold:
\begin{equation*}
\begin{cases}
((2\alpha^2-1)^2-\alpha^2)((2\alpha^2-1)^2-\beta^2)=0\\
((2\beta^2-1)^2-\alpha^2)((2\beta^2-1)^2-\beta^2)=0\\
2\alpha\beta((2\alpha\beta)^2-(2\alpha)^2)((2\alpha\beta)^2-(2\beta)^2)((2\alpha\beta)^2-(\alpha+\beta)^2)((2\alpha\beta)^2-(\alpha-\beta)^2)=0.
\end{cases}
\end{equation*}
This admits the unique feasible solution $\alpha=(-1+\sqrt5)/4$ and $\beta=(1+\sqrt5)/4$.
\end{proof}
Recall that a binary code of length $d$ with minimum distance $\Delta$ is a set $\mathcal{B}\subseteq\mathbb{F}_2^d$ such that $\mydist(b,b')\geq \Delta$ for every distinct $b,b'\in\mathcal{B}$ where $\mydist(.,.)$ denotes the Hamming distance \cite{EZBOOK}. Applying the following function 
\begin{equation*}
\Sigma\colon \mathbb{F}_2\mapsto \mathbb{R},\qquad \Sigma(0)=1/\sqrt{d},\qquad \Sigma(1)=-1/\sqrt{d}
\end{equation*}
entrywise on the codewords (i.e., on the elements of $\mathcal{B}$) yields a spherical embedding.
\begin{lemma}\label{LCODE}
Let $d\geq 2$, and let $\Delta_1,\Delta_2\in\{1,\dots,d-1\}$. Let $\mathcal{B}$ be a binary code of length $d$, such that $\mydist(b,b')\in\{\Delta_1,\Delta_2,d-\Delta_1,d-\Delta_2\}$ for every distinct $b,b'\in\mathcal{B}$. Then $\mathcal{X}:=\{\Sigma(b)\colon b\in\mathcal{B}\}$ spans a system of biangular lines with $A(\mathcal{X})\subseteq\{\pm(1-2\Delta_1/d), \pm(1-2\Delta_2/d)\}$.
\end{lemma}
\begin{proof}
For every $b,b'\in\mathcal{B}$, we have $\left\langle \Sigma(b),\Sigma(b')\right\rangle = 1-2\mydist(b,b')/d>-1$.
\end{proof}
For terminology and basic facts on lattices we refer the reader to the textbook \cite{SPLAG}. It is well-known (see \cite[p.\ 117]{SPLAG}) that the shortest vectors of the $D_d$ lattices give rise to biangular line systems.
\begin{lemma}\label{L1}
Let $d\geq 2$, and let $\mathcal{X}\subset\mathbb{R}^d$ be the subset of all permutations of the unit vectors $[\pm1,\pm1,0,\dots,0]/\sqrt 2$ whose first nonzero coordinate is positive. Then $\mathcal{X}$ spans $|\mathcal{X}|=d(d-1)$ biangular lines with $A(\mathcal{X})\subseteq\{0,\pm 1/2\}$.
\end{lemma}
\begin{proof}
For distinct $x,x'\in\mathcal{X}$ the inner product $\left\langle x,x'\right\rangle$ depends on the number of positions where the nonzero coordinates of $x$ and $x'$ overlap. If there is no overlap, or there are exactly two overlaps, then $\left\langle x,x'\right\rangle=0$. Otherwise, if there is a single overlap, then $\left\langle x,x'\right\rangle=\pm1/2$.
\end{proof}
\begin{remark}\label{REM1}
We remark that for $d\in\{6,7,8\}$ the set of (nonantipodal) shortest vectors of the exceptional lattices $E_d$ give rise to biangular line systems in $\mathbb{R}^d$ with inner product set $\{0,\pm1/2\}$ formed by $36$, $63$, and $120$ lines, respectively \cite[p.~120]{SPLAG}.
\end{remark}
Let $h>0$, $h<1$. Starting from a spherical $2$-distance set $\mathcal{X}\subset\mathbb{R}^d$, one may obtain a family of biangular line systems in $\mathbb{R}^{d+1}$, where the vectors $x\in\mathcal{X}$ are rescaled by a factor of $\sqrt{1-h^2}$ and translated along the $(d+1)$th coordinate to height $h$. In a similar spirit, the $6$ diagonals of the icosahedron can be continuously twisted in $\mathbb{R}^3$, yielding a family of biangular lines \cite{BTF}.
\begin{proposition}[Infinite families]\label{PINF}
Let $\mathcal{X}\subset\mathbb{R}^d$ be a spherical $2$-distance set with $A(\mathcal{X})\subseteq\{\alpha,\beta\}$, with $\alpha,\beta\geq-1$ and $\alpha,\beta<1$. Let $h>0, h <1$. Then
\begin{equation*}
\mathcal{Y}(h):=\{[h,\sqrt{1-h^2}x]\colon x\in\mathcal{X}\}
\end{equation*}
spans a system of biangular lines in $\mathbb{R}^{d+1}$ with  $A(\mathcal{Y}(h))\subseteq\{h^2+(1-h^2)\alpha,h^2+(1-h^2)\beta\}$.
\end{proposition}
\begin{proof}
For every $y,y'\in\mathcal{Y}(h)$ we have $\left\langle y,y'\right\rangle=h^2+(1-h^2)\left\langle x,x'\right\rangle$ for some $x,x'\in\mathcal{X}$. Furthermore, $-1\not\in A(\mathcal{Y}(h))$ by our assumptions on $h$.
\end{proof}
Since the midpoints of the edges of the regular simplex in $\mathbb{R}^d$ forms a spherical $2$-distance set of size $d(d+1)/2$, biangular lines of this cardinality are abundant in $\mathbb{R}^{d+1}$. Translation to a well-chosen height yields the following variant.
\begin{proposition}[Lifting]\label{P2}
Let $\mathcal{X}\subset\mathbb{R}^d$ be a spherical $4$-distance set with $A(\mathcal{X})\subseteq\{\alpha,\beta,\gamma$, $\alpha+\beta-\gamma\}$, with $\alpha,\beta,\gamma\geq-1$ and $\alpha,\beta,\gamma<1$, and assume that $\alpha+\beta<0$. Then $\mathcal{Y}:=\{[\sqrt{-\alpha-\beta},\sqrt2x]/\sqrt{2-\alpha-\beta}\colon x\in\mathcal{X}\}$ spans a system of biangular lines in $\mathbb{R}^{d+1}$ with $A(\mathcal{Y})\subseteq\{\pm\frac{\alpha-\beta}{2-\alpha-\beta},\pm\frac{2\gamma-\alpha-\beta}{2-\alpha-\beta}\}$.
\end{proposition}
\begin{proof}
For every $y,y'\in\mathcal{Y}$ we have $\left\langle y,y'\right\rangle=(-(\alpha+\beta)+2\left\langle x,x'\right\rangle)/(2-\alpha-\beta)$ for some $x,x'\in\mathcal{X}$. Furthermore, $-1\not\in A(\mathcal{Y})$ by our assumptions on $\alpha,\beta,\gamma$.
\end{proof}
\begin{remark}
Given a spherical $3$-distance set $\mathcal{X}$ with $A(\mathcal{X})\subseteq\{\alpha,\beta,\gamma\}$, then it might happen that $\alpha+\beta<0$, $\alpha+\gamma<0$, and $\beta\neq\gamma$. When this occurs, lifting via Proposition~\ref{P2} could result in nonisometric biangular line systems.
\end{remark}
The main utility of Proposition~\ref{P2} is that antipodal vectors (spanning the very same lines) can be split into two nonantipodal vectors one dimension higher. It immediately follows that any equiangular line system leads to twice as many biangular lines in one dimension higher.
\begin{theorem}\label{T1}
For every $d\geq 3$, there exists a set $\mathcal{X}\subset\mathbb{R}^d$ spanning $|\mathcal{X}|=2(d-1)(d-2)$ biangular lines with $A(\mathcal{X})\subseteq\{\pm1/5,\pm3/5\}$.
\end{theorem}
\begin{proof}
Take all $2(d-1)(d-2)$ vectors in $\mathbb{R}^{d-1}$ forming a spherical $4$-distance set $\mathcal{X}$ with $A(\mathcal{X})\subseteq \{-1,-1/2,0,1/2\}$ in Lemma~\ref{L1} and then use Proposition~\ref{P2} to get the claimed biangular line systems.
\end{proof}
A further application of Proposition~\ref{P2} is the following.
\begin{corollary}\label{CUBICCOR}
For $d\in\{4,\dots,16\}$ there exists a set $\mathcal{X}\subset\mathbb{R}^d$ spanning $|\mathcal{X}|=\binom{d}{3}$ biangular lines. There exists a set $\mathcal{X}\subset\mathbb{R}^{17}$ spanning $|\mathcal{X}|=\binom{18}{3}$ biangular lines.
\end{corollary}
\begin{proof}
Consider the `canonical' spherical $3$-distance set
$\mathcal{X}:=\{\text{All permutations of}$ $ \sqrt{\frac{d-3}{3d}}[1,1,1,-\frac{3}{d-3},\dots,-\frac{3}{d-3}]\in\mathbb{R}^{d}\}$ of cardinality $\binom{d}{3}$ with $A(\mathcal{X})\subseteq\{-\frac{3}{d-3},\frac{d-9}{3(d-3)}$, $\frac{2d-9}{3(d-3)}\}$. Since for every $x\in\mathcal{X}$, $\left\langle x,[1,1,\dots,1]\right\rangle=0$, $\mathcal{X}$ is embedded into $\mathbb{R}^{d-1}$. Consequently if $d=18$, then $\mathcal{X}$ spans a biangular line system in $\mathbb{R}^{17}$. If $d\leq 16$, then since $\frac{d-9}{3(d-3)}-\frac{3}{d-3}<0$, Proposition~\ref{P2} yields the claimed configurations in $\mathbb{R}^d$.
\end{proof}
Finally, a rather surprising consequence of Proposition~\ref{P2} is the following: the biangular line systems mentioned in Remark~\ref{REM1} are not the best possible in their respective dimension.
\begin{corollary}\label{C1}
There exists a set $\mathcal{X}\subset\mathbb{R}^d$ spanning biangular lines with $A(\mathcal{X})\subseteq\{\pm1/5,\pm3/5\}$ for $(d,|\mathcal{X}|)\in\{(3,4),(4,12),(5,24),(6,40),(7,72),(8,126), (9,240)\}$.
\end{corollary}
\begin{proof}
The cases $d\in\{3,4,5,6\}$ follow from Theorem~\ref{T1}. To see the remaining cases, combine Proposition~\ref{P2} with the exceptional configurations mentioned in Remark~\ref{REM1}.
\end{proof}
Later (see Section~\ref{sec4}) we will show that Theorem~\ref{T1} gives rise to a largest possible line system for $d\in\{4,5,6\}$, and we tend to believe that Corollary~\ref{C1} gives the best configurations for $d\in\{7,8,9\}$ as well.

Next we prove a preliminary technical result. Following the terminology of \cite{LS}, we denote by $N_{1/3}(d)$ the maximum number of equiangular lines in $\mathbb{R}^d$ where the set of inner products is a subset of $\{\pm1/3\}$. We note that $N_{1/3}(0)=0$.
\begin{proposition}\label{P3}
For $m\geq 1$ and $d\geq m$, there exists a set $\mathcal{X}\subset\mathbb{R}^d$ spanning $|\mathcal{X}|=2m\cdot N_{1/3}(d-m)$ biangular lines with $A(\mathcal{X})\subseteq\{\pm1/5,\pm3/5\}$.
\end{proposition}
\begin{proof}
Let $\mathcal{E}$ denote the set of canonical basis vectors of $\mathbb{R}^{m}$, and consider a maximum set $\mathcal{Y}\subset\mathbb{R}^{d-m}$ spanning $N_{1/3}(d-m)$ equiangular lines with $A(\mathcal{Y})\subseteq\{\pm1/3\}$. We claim that the following set $\mathcal{X}\subset\mathbb{R}^d$ spans a biangular line system:
\begin{equation*}
\mathcal{X}:=\{[\sqrt6y,2e]/\sqrt{10}\colon y\in\mathcal{Y}, e\in\mathcal{E}\}\cup\{[\sqrt6y,-2e]/\sqrt{10}\colon y\in\mathcal{Y}, e\in\mathcal{E}\}.
\end{equation*}
Indeed, as for $x,x'\in\mathcal{X}$, we have $\left\langle x,x'\right\rangle = 3\left\langle y,y'\right\rangle/5\pm2\left\langle e,e'\right\rangle/5$ for some (not necessarily distinct) $e,e'\in\mathcal{E}$ and $y,y'\in\mathcal{Y}$. Since $\left\langle e,e'\right\rangle\in\{0,1\}$ and $\left\langle y,y'\right\rangle\in\{\pm1/3,1\}$ the claim follows.
\end{proof}
We note the following.
\begin{corollary}
There exists a set $\mathcal{X}\subset\mathbb{R}^{14}$ spanning $|\mathcal{X}|=392$ biangular lines with $A(\mathcal{X})\subseteq\{\pm1/5,\pm3/5\}$.
\end{corollary}
\begin{proof}
Follows from Proposition~\ref{P3} by setting $m=7$ and $d=7$, and by recalling from \cite{LS} that $N_{1/3}(7)=28$.
\end{proof}
It turns out that one may combine certain line systems described in Proposition~\ref{P3} with the $256$ lines spanned by the `even half' of the $10$ dimensional hypercube. This yields improved results for $d\in\{10,11,12,13,15,16\}$, and gives the same number of biangular lines for $d=17$ as Corollary~\ref{CUBICCOR}.
\begin{theorem}\label{T2}
For $d\geq 10$, there exists a set $\mathcal{X}\subset\mathbb{R}^d$ spanning $|\mathcal{X}|=256+20N_{1/3}(d-10)$ biangular lines with $A(\mathcal{X})\subseteq\{\pm1/5,\pm3/5\}$.
\end{theorem}
\begin{proof}
Let $\mathcal{B}\subset\mathbb{F}_2^{10}$ be the binary code of length $10$ formed by codewords of even weight, such that the first coordinate of every $b\in\mathcal{B}$ is $0$. By Lemma~\ref{LCODE} the set $\mathcal{Z}:=\{\Sigma(b)\colon b\in\mathcal{B}\}\subset\mathbb{R}^{10}$ spans a system of $256$ biangular lines with $A(\mathcal{Z})\subseteq\{\pm1/5,\pm3/5\}$. Next, we consider a maximum set $\mathcal{Y}\subset\mathbb{R}^{d-10}$ spanning $N_{1/3}(d-10)$ equiangular lines with $A(\mathcal{Y})\subseteq\{\pm1/3\}$. Let $\mathcal{E}$ denote the set of canonical basis vectors of $\mathbb{R}^{10}$, and let $o\in\mathbb{R}^{d-10}$ denote the zero vector. We claim that the following set $\mathcal{X}\subset\mathbb{R}^{d}$ spans a biangular line system:
\begin{multline*}
	\mathcal{X}:=\{[\sqrt6y,2e]/\sqrt{10}\colon y\in\mathcal{Y}, e\in\mathcal{E}\}\\
	{}\cup\{[\sqrt6y,-2e]/\sqrt{10}\colon y\in\mathcal{Y}, e\in\mathcal{E}\}\cup\{[o,z]\colon z\in\mathcal{Z}\}.
\end{multline*}
Indeed, since for $x,x'\in\mathcal{X}$, we have
\begin{equation*}
\left\langle x,x'\right\rangle\in\{3\left\langle y,y'\right\rangle/5\pm2\left\langle e,e'\right\rangle/5,\pm2\left\langle e,z\right\rangle/\sqrt{10},\left\langle z,z'\right\rangle\}
\end{equation*}
for some (not necessarily distinct) $e,e'\in\mathcal{E}$, $y,y'\in\mathcal{Y}$, and $z,z'\in\mathcal{Z}$. Since $\left\langle e,e'\right\rangle\in\{0,1\}$, $\left\langle e,z\right\rangle\in\{\pm1/\sqrt{10}\}$, $\left\langle y,y'\right\rangle\in\{\pm1/3,1\}$, and $\left\langle z,z'\right\rangle\in\{\pm1/5,\pm3/5,1\}$ the claim follows.
\end{proof}
\begin{corollary}
There exists a set $\mathcal{X}\subset\mathbb{R}^d$ spanning biangular lines with $A(\mathcal{X})\subseteq\{\pm1/5,\pm3/5\}$ for
\begin{equation*}
(d,|\mathcal{X}|)\in\{(10,256), (11,276), (12,296), (13,336), (15,456), (16,576), (17,816)\}.
\end{equation*}
\end{corollary}
\begin{proof}
Combine Theorem~\ref{T2} with \cite[Theorem~4.5]{LS}.
\end{proof}
Finally, we note that various cross-sections of the Leech lattice $\Lambda_{24}$ (see \cite[p.\ 133]{SPLAG} for how to construct its shortest vectors from the extended binary Golay code \cite{AEB} in explicit form) gives rise to biangular line systems with inner product set $\{0,\pm1/3\}$.
\begin{theorem}\label{T3}
There exists a set $\mathcal{X}\subset\mathbb{R}^d$ spanning biangular lines with $A(\mathcal{X})\subseteq\{0,\pm1/3\}$ for $(d,|\mathcal{X}|)\in\{(21,896),(22,1408),(23,2300)\}$.
\end{theorem}
\begin{proof}
Let $\mathcal{L}\subset\mathbb{R}^{24}$, $|\mathcal{L}|=196560$, be the set of shortest vectors of $\Lambda_{24}$, where the vectors are normalized so that $\left\langle \ell,\ell\right\rangle=1$ for every $\ell\in\mathcal{L}$. With this convention, $\left\langle \ell, \ell'\right\rangle\in\{0,\pm 1/4,\pm1/2,\pm1\}$ for every $\ell,\ell'\in\mathcal{L}$. Now let $\ell\in\mathcal{L}$ be fixed. It is well-known (see \cite[p.\ 264]{SPLAG}) that the following subset $\mathcal{Y}=\{y\colon\left\langle \ell,y\right\rangle = 1/2; y\in\mathcal{L}\}$ contains $4600$ vectors, independently of the choice $\ell$. Note that for $y\in\mathcal{Y}$ we have $\ell-y\in\mathcal{Y}$ and therefore the set $\mathcal{Z}:=\{(2y-\ell)/\sqrt3\colon y\in\mathcal{Y}\}$ is antipodal, and $\left\langle \ell,z\right\rangle=0$ for every $z\in\mathcal{Z}$. Finally, let $\mathcal{X}\subset\mathcal{Z}$ with $|\mathcal{X}|=2300$ so that $\mathcal{Z}=\{x\colon x\in\mathcal{X}\}\cup\{-x\colon x\in\mathcal{X}\}$. Now $\mathcal{X}$ spans the claimed biangular line system in $\mathbb{R}^{23}$, since $\left\langle y,y'\right\rangle\not\in\{-1/4,-1\}$ and therefore for $x,x'\in\mathcal{X}$ we have $\left\langle x,x'\right\rangle=(4\left\langle y,y'\right\rangle-1)/3\in\{0,\pm1/3,1\}$. Let $x,x'\in \mathcal{X}$ so that $\left\langle x,x'\right\rangle = 0$. Then the cross-sections $\mathcal{U}:=\{u\colon \left\langle u,x\right\rangle = 0; u\in\mathcal{X}\}$, $\mathcal{V}:=\{v\colon \left\langle v,x\right\rangle = \left\langle v,x'\right\rangle = 0; v\in\mathcal{X}\}$ span the claimed biangular line systems in dimension $22$ and $21$, respectively.
\end{proof}
Another way to get biangular lines with set of inner products $\{0,\pm1/3\}$ is the following.
\begin{lemma}\label{LSPEC}
Let $w\equiv3\ (\mymod 4)$ and $d\geq 2w+1$ be positive integers. Let $\mathcal{B}\subset\mathbb{F}_2^d$ be a binary constant weight code of length $d$, weight $w$ and minimum distance $2w-2$, and assume that there exists a Hadamard matrix of order $w+1$. Then there exists a set $\mathcal{X}\subset\mathbb{R}^d$ with $|\mathcal{X}|=(w+1)|\mathcal{B}|$ spanning a biangular line system with $A(\mathcal{X})\subseteq\{0,\pm1/w\}$.
\end{lemma}
\begin{proof}
Recall that a Hadamard matrix $H$ of order $w+1$ is a $(w+1)\times(w+1)$ orthogonal matrix with entries $\pm1/\sqrt{w+1}$. Let $H'$ be the matrix obtained from $H$ after removing its first column, and renormalizing its rows. Let $\mathcal{H}\subset\mathbb{R}^{w}$ be the set of rows of $H'$. Clearly, $\left\langle h,h'\right\rangle \in\{\pm1/w,1\}$ for $h,h'\in\mathcal{H}$. Now $\mathcal{X}$ can be obtained by replacing each codeword $b\in\mathcal{B}$ with a set of $w+1$ real vectors where the support of $b$ (i.e., coordinates with binary $1$) are replaced by the entries of $h\in\mathcal{H}$, and coordinates with binary $0$ are replaced by $0\in\mathbb{R}$. Since $d\geq 2w+1$, there are no two codewords at Hamming distance $d$, and therefore the claim follows.
\end{proof}
\begin{corollary}\label{conseq}
For $d\geq 7$ there exists a set $\mathcal{X}\subset\mathbb{R}^d$ spanning $|\mathcal{X}|=4\left\lceil(d-1)(d-2)/6\right\rceil$ biangular lines with $A(\mathcal{X})\subseteq\{0,\pm1/3\}$. Furthermore, there exists a set $\mathcal{Y}\subset\mathbb{R}^{d+1}$ spanning $|\mathcal{X}|$ biangular lines with $A(\mathcal{Y})\subseteq\{\pm1/7,\pm3/7\}$.
\end{corollary}
\begin{proof}
Indeed, this is a specialization of Lemma~\ref{LSPEC} for $w=3$ and using constant weight codes coming from the averaging argument in \cite[Theorem~14]{IEEECODES}. The second part of the claim is an immediate consequence of Proposition~\ref{P2}.
\end{proof}
While Corollary~\ref{conseq} is weaker than Theorem~\ref{T1}, it can be used in two ways. First, one may embed the $2300$ biangular lines from Theorem~\ref{T3} into $\mathbb{R}^{23+d}$, and extend this configuration with an additional $4\left\lceil(d-1)(d-2)/6\right\rceil$ vectors (for $d\geq 7$). Secondly, it may happen that these configurations can be further extended to a spherical $4$-distance set with inner products $\{-2/3,-1/3,0,1/3\}$, and then an application of Proposition~\ref{P2} would immediately yield biangular lines with inner products $\{\pm 1/7,\pm3/7\}$ in $\mathbb{R}^{24+d}$. One consequence of the following result is that the two largest sets mentioned in Theorem~\ref{T3} are inextendible.
\begin{theorem}[The relative bound, \cite{DGS}]\label{RELBDS}
Let $d\geq 3$, and assume that $\mathcal{X}\subset\mathbb{R}^d$ spans a biangular line system with $A(\mathcal{X})\subseteq\{\pm\alpha,\pm\beta\}$, $0\leq \alpha,\beta <1$. Assume that $\alpha^2+\beta^2\leq 6/(d+4)$. Let $n_{\alpha}:=|\{[x,x']\colon \left\langle x,x'\right\rangle^2=\alpha^2;x,x'\in\mathcal{X}\}|$. Then
\begin{equation}\label{RELBD}
|\mathcal{X}|\leq \frac{d(d+2)(1-\alpha^2)(1-\beta^2)}{3-(d+2)(\alpha^2+\beta^2)+d(d+2)\alpha^2\beta^2}
\end{equation}
if the denominator of the right hand side is positive. Equality holds if and only if
\begin{equation*}
\begin{cases}
(\frac{6}{d+4}-\alpha^2-\beta^2)((\alpha^2-\beta^2)n_\alpha +|\mathcal{X}|(|\mathcal{X}|-1)\beta^2+|\mathcal{X}|-\frac{|\mathcal{X}|^2}{d})=0\\
(\frac{6}{d+4}-\alpha^2-\beta^2)(\alpha^2-\beta^2)n_\alpha=\frac{|\mathcal{X}|(d^2+3|\mathcal{X}|-4)}{(d+2)(d+4)}-|\mathcal{X}|(|\mathcal{X}|-1)\beta^2(\frac{6}{d+4}-\beta^2).
\end{cases}
\end{equation*}
\end{theorem}
\begin{proof}
This result is well-known \cite{BOY}, \cite{DGS}. Equality holds if and only if  $(\frac{6}{d+4}-\alpha^2-\beta^2)\sum_{x,x'\in\mathcal{X}}C_2^{((d-2)/2)}(\left\langle x,x'\right\rangle)=\sum_{x,x'\in\mathcal{X}}C_4^{((d-2)/2)}(\left\langle x,x'\right\rangle)=0$, where $C_i^{(j)}(z)$ denotes the Gegenbauer polynomials (see \cite{DGS2}).
\end{proof}
\begin{remark}
If $\mathcal{X}$ forms a spherical $4$-design \cite{BAN9}, then equality holds in \eqref{RELBD}.
\end{remark}
\begin{remark}
If there is equality in \eqref{RELBD}, then the quantity $n_{\alpha}$ as defined in Theorem~\ref{RELBDS} is a nonnegative integer. The failure of this condition could be used to show the nonexistence of various hypothetical configurations. In particular, in $\mathbb{R}^8$ there does not exist $50$ biangular lines with set of inner products $\{\pm1/4,\pm1/2\}$.
\end{remark}
In Table~\ref{TABREL} we display data on the known biangular line systems meeting the relative bound, and later in Corollary~\ref{RELBDCOR} we prove that this list is (essentially) complete for $d\leq 6$. The canonical examples are mutually unbiased bases \cite{KREDOCK}, spanning $2^{4i-1}+2^{2i}$ biangular lines in dimension $d=4^{i}$ with inner products $\{0,\pm2^{-i}\}$, $i\geq1$. We believe that the following example is new.
\begin{example}[$36$ biangular lines in $\mathbb{R}^7$ with set of inner products $\{\pm1/7,\pm3/7\}$]
Let $U$ be the $7\times 7$ circulant matrix with first row $[0,1,0,0,0,0,0]$. Let $\mathcal{Y}:=\{[-7,1,1,1,1,1,1]$, $[-1, 3, 3, -3, 3, -3, -3]\}$, and let $\mathcal{Z}:=\{[1, -1, -3, 3, 3, -3, -3]$, $[1, 3, -1, -3, -3, -3, 3]$, $[-1, 3, -3, 1, -3, 3, -3]\}$.
Then, the following set
\begin{multline*}
\mathcal{X}=\{[-7,1,1,1,1,1,1,1]/\sqrt{56}\}\cup\{[1,yU^i]/\sqrt{56}\colon i\in\{0,1,\dots,6\};y\in\mathcal{Y}\}\\
{}\cup\{[3,zU^i]/\sqrt{56}\colon i\in\{0,1,\dots,6\};z\in\mathcal{Z}\}
\end{multline*}
spans $36$ biangular lines in $\mathbb{R}^7$ with $A(\mathcal{X})\subseteq\{\pm1/7,\pm3/7\}$. Indeed, all vectors are orthogonal to $[1,\dots,1]\in\mathbb{R}^8$. The parameters of this line system meet the relative bound.\qed
\end{example}
Despite our best efforts, we were unable to find any references to the following example.
\begin{example}[$256$ biangular lines in $\mathbb{R}^{16}$ with set of inner products $\{0,\pm1/3\}$]
Consider a biplane \cite{BIPLANEPAT} of order $4$, that is a $16\times 16$ square $\{0,1\}$-matrix $H$ with constant row and column sum $6$, such that $HH^T=4I_{16}+2J_{16}$. We may simply take $H:=(J_4-I_4)\otimes I_4+I_4\otimes(J_4-I_4)$, and let $\mathcal{H}\subset\mathbb{R}^{16}$ be the set of rows of $H$. Let $\mathcal{B}\subset\mathbb{F}_2^{6}$ be a binary code of length $6$ formed by codewords of even weight, such that the first coordinate of every $b\in\mathcal{B}$ is $0$. By Lemma~\ref{LCODE} the set $\mathcal{Z}:=\{\Sigma(b)\colon b\in\mathcal{B}\}\subset\mathbb{R}^{6}$ spans a system of $16$ biangular lines with $A(\mathcal{Z})\subseteq\{\pm1/3\}$. Replacing each codeword $b\in\mathcal{B}$ with a set of $16$ real vectors where the support of $b$ (i.e., coordinates with binary $1$) are replaced by the entries of $h\in\mathcal{H}$, and coordinates with binary $0$ are replaced by $0\in\mathbb{R}$ spans the claimed $256$ biangular lines in $\mathbb{R}^{16}$. The parameters of this line system meet the relative bound.\qed
\end{example}
\begin{table}[htbp]%
	\centering
	\caption{Biangular line systems meeting the relative bound}\label{TABREL}
	\tiny\begin{tabular}{llll}
		\toprule
		$d$ & $n$ & $\{\alpha,\beta\}$ & \text{Remark}\\
		\midrule
		3   &    6 & $\{\pm1/\sqrt5\}$ & \text{Icosahedron}\\
		&   10 & $\{\pm1/3,\pm\sqrt5/3\}$ & \text{Dodecahedron}\\
		4   &   12 & $\{0,\pm1/2\}$ & \text{$D_4$ lattice (MUBs)}\\
		6   &   27 & $\{\pm1/4,\pm1/2\}$ & \text{Schl\"afli graph}\\
		&   27 & $\{\pm1/4,\pm1/2\}$ & \text{Example~\ref{EX3}}\\
		&   36 & $\{0,\pm1/2\}$ & \text{$E_6$ lattice}\\
		7   &   28 & $\{\pm1/3\}$   & \text{Equiangular lines}\\
		&   36 & $\{\pm1/7,\pm3/7\}$ & \text{Example~\ref{EX1}}\\
		&   63 & $\{0,\pm1/2\}$ &  \text{$E_7$ lattice}\\
		8   &  120 & $\{0,\pm1/2\}$ &  \text{$E_8$ lattice}\\    
		16  &  144 & $\{0,\pm1/4\}$ & \text{MUBs}\\
		&  256 & $\{0,\pm1/3\}$            & \text{From a biplane of order $4$}\\
		22  &  275 & $\{\pm1/6,\pm1/4\}$ & \text{McLaughlin graph}\\
		& 1408 & $\{0,\pm1/3\}$ & \text{From $\Lambda_{24}$}\\
		23  &  276 & $\{\pm1/5\}$   & \text{Equiangular lines}\\
		& 2300 & $\{0,\pm1/3\}$ & \text{From $\Lambda_{24}$}\\
		$4^i$ & $2^{4i-1}+2^{2i}$ & $\{0,\pm2^{-i}\}$ & MUBs, $i\geq 3$\\
		\bottomrule
	\end{tabular}
\end{table}
Finally, we note the following (almost immediate) consequence of \cite[Theorem~5.2 and 5.3]{NOZ}.
\begin{theorem}[See \cite{NOZ}]
Let $d\geq 5$, and let $\mathcal{X}\subset\mathbb{R}^d$ span a maximum biangular line system with $A(\mathcal{X})\subseteq\{\pm\alpha,\pm\beta\}$, $0\leq \alpha < \beta < 1$. Then $z:=(1-\alpha^2)/(\beta^2-\alpha^2)$ is an integer. Furthermore $z\leq\left\lfloor 1/2+\sqrt{(d^2+d+2)(d^2+d-1)/(4d^2+4d-8)}\right\rfloor$.
\end{theorem}
\begin{proof}
The statement is a reformulation of \cite[Theorem~5.2 and 5.3]{NOZ} and it holds whenever $|\mathcal{X}|\geq d(d+1)$. This in turn holds by Theorem~\ref{T1} for maximum biangular line systems whenever $d\geq 7$. For $d\in\{5,6\}$ the set of inner products of (the unique) maximum biangular line systems is $\{\pm1/5,\pm3/5\}$ (see Theorem~\ref{THMDIM5} and \ref{THMDIM6}), and therefore in these cases $z=3$ is indeed an integer below the claimed bound.
\end{proof}
%%%%%%%%%%%%%%%%%%%%%%%%%%%%%%%%%%%%%%
% SECTION 3: COMPUTATIONAL FRAMEWORK %
%%%%%%%%%%%%%%%%%%%%%%%%%%%%%%%%%%%%%%
\section{Computational framework}\label{sec3}
In this section, following ideas developed in \cite{SZO}, we set up a framework for systematically generating biangular lines. We will leverage on this newly established theory in Section~\ref{sec4} where we demonstrate how to use this approach in practice. In particular, we will determine the size of the largest biangular line systems in dimension $d\leq 6$ by using supercomputational resources, and classify the maximum cases.

We remark that this framework carries over to the multiangular setting after minor technical changes (see Section~\ref{sec5} and Appendix~\ref{APPENDIX1}).
\subsection{A high level overview}
Let $d,n\geq1$, and let $\mathcal{X}=\{x_1,\dots,x_n\}\subset\mathbb{R}^d$ be a set of unit vectors, spanning a system of $n$ biangular lines. Starting from this section, we will represent $\mathcal{X}$ by its Gram matrix $G:=[\left\langle x_i,x_j\right\rangle]_{i,j=1}^n$. Conveniently, the matrix $G$ is invariant up to change of basis, and has the following combinatorial properties: $G$ is of $n\times n$; $G=G^T$; $G_{ii}=1$ for every $i\in\{1,\dots,n\}$, and $G_{i,j}\in A(\mathcal{X})$ for distinct $i,j\in\{1,\dots,n\}$. Furthermore, it has the following algebraic properties: $G$ is positive-semidefinite; and $\myrank G \leq d$. Conversely, starting from any matrix $G$ having these properties, one may reconstruct an $n\times \myrank{G}$ matrix $F$ (uniquely, up to change of basis) via the Cholesky decomposition so that $FF^T=G$ holds \cite{NJH}.

Our aim is to find a way for generating all (sufficiently large) $n\times n$ Gram matrices of biangular line systems in a fixed dimension $d$. It follows from Ramsey theory that $n$ is bounded in terms of $d$, and we recall here the following explicit bound.
\begin{theorem}[Absolute bound, \cite{DGS}]\label{T4}
Let $\mathcal{X}\subset\mathbb{R}^d$ span a biangular line system. Then $|\mathcal{X}|\leq \binom{d+3}{4}$.
\end{theorem}
We say that the permutation $\sigma$ of the set $\Gamma=\{\alpha,\beta,-\alpha,-\beta\}$ is a relabeling if $\sigma(\gamma)=-\sigma(-\gamma)$ for every $\gamma\in \Gamma$. The following concept is central to this paper.
\begin{definition}\label{MAINDEF}
Let $C(\alpha,\beta)$ be an $n\times n$ symmetric matrix with constant diagonal $1$ over the polynomial ring $\mathbb{Q}[\alpha,\beta]$ whose off-diagonal entries are $\{0,\pm\alpha,\pm\beta\}$. Two such matrices, $C_1$ and $C_2$ are called equivalent, if $C_1(\alpha,\beta)=PC_2(\sigma(\alpha),\sigma(\beta))P^T$ for some signed permutation matrix $P$ and relabeling $\sigma$. A representative of this matrix equivalence class is called a candidate Gram matrix.\qed
\end{definition}
Candidate Gram matrices capture the combinatorial structure of Gram matrices. Since our focus is on the biangular case, we will assume in the following that
\begin{equation}\label{ASSUME}
\alpha\beta(\alpha^2-\beta^2)(\alpha^2-1)(\beta^2-1)\neq 0.
\end{equation}
Furthermore, at most two out of the three symbols $0$, $\pm\alpha$, $\pm\beta$ can appear as a matrix entry in $C(\alpha,\beta)$. Clearly, if $G$ is a Gram matrix of a biangular line system, then there exist a candidate Gram matrix $C(\alpha,\beta)$, such that $G=C(\alpha^\ast,\beta^\ast)$ for some $\alpha^\ast,\beta^\ast\in\mathbb{R}$, subject to \eqref{ASSUME}. In particular, $\myrank C(\alpha^\ast,\beta^\ast)\leq d$ should hold.
\begin{example}[The candidate Gram matrices of order $3$]\label{EX1}
\[\left\{\left[\begin{smallmatrix}
1 & 0 & 0\\
0 & 1 & 0\\
0 & 0 & 1\\
\end{smallmatrix}\right], \qquad \left[\begin{smallmatrix}
1 & 0 & 0\\
0 & 1 & \alpha\\
0 & \alpha & 1\\
\end{smallmatrix}\right], \qquad \left[\begin{smallmatrix}
1 & 0 & \alpha\\
0 & 1 & \alpha\\
\alpha & \alpha & 1\\
\end{smallmatrix}\right], \qquad \left[\begin{smallmatrix}
1 & \alpha & \alpha\\
\alpha & 1 & \alpha\\
\alpha & \alpha & 1\\
\end{smallmatrix}\right], \qquad \left[\begin{smallmatrix}
1 & \alpha & \alpha\\
\alpha & 1 & \beta\\
\alpha & \beta & 1\\
\end{smallmatrix}\right]\right\}\]
Note that at most two symbols appear (whose values are unspecified) within the off-diagonal positions, signifying distinct inner products.\qed
\end{example}
The main advantage of using candidate Gram matrices is that in this way we are transforming the problem of `infinitely many $n\times n$ Gram matrices' to the conceptually simpler `finite list of $n\times n$ candidate Gram matrices' (where $n$ itself is bounded by Theorem~\ref{T4}). Then, one should decide whether a candidate Gram matrix actually represents a Gram matrix via a spectral analysis, as illustrated below.
\begin{example}[The Petersen graph and related structures, cf.~Proposition~\ref{PINF}]\label{EXPET}
Consider the following example of a candidate Gram matrix of order $10$:
\begin{equation*}
C(\alpha,\beta)=\left[
\begin{smallmatrix}
1 & \alpha & \alpha & \alpha & \alpha & \alpha & \alpha & \beta & \beta & \beta \\
\alpha & 1 & \alpha & \alpha & \alpha & \beta & \beta & \alpha & \alpha & \beta \\
\alpha & \alpha & 1 & \alpha & \beta & \alpha & \beta & \alpha & \beta & \alpha \\
\alpha & \alpha & \alpha & 1 & \beta & \beta & \alpha & \beta & \alpha & \alpha \\
\alpha & \alpha & \beta & \beta & 1 & \alpha & \alpha & \alpha & \alpha & \beta \\
\alpha & \beta & \alpha & \beta & \alpha & 1 & \alpha & \alpha & \beta & \alpha \\
\alpha & \beta & \beta & \alpha & \alpha & \alpha & 1 & \beta & \alpha & \alpha \\
\beta & \alpha & \alpha & \beta & \alpha & \alpha & \beta & 1 & \alpha & \alpha \\
\beta & \alpha & \beta & \alpha & \alpha & \beta & \alpha & \alpha & 1 & \alpha \\
\beta & \beta & \alpha & \alpha & \beta & \alpha & \alpha & \alpha & \alpha & 1 \\
\end{smallmatrix}
\right].
\end{equation*}
Here $C(0,1)-I_{10}$ is the adjacency matrix of the Petersen graph. Using standard spectral graph theory, one may find that for every $\alpha^\ast, \beta^\ast\in\mathbb{R}$ we have $\Lambda(C(\alpha^\ast,\beta^\ast))=\{[1+6\alpha^\ast+3\beta^\ast]^1,[1+\alpha^\ast-2\beta^\ast]^4,[1-2\alpha^\ast+\beta^\ast]^5\}$. Therefore $\myrank C(\alpha^\ast,2\alpha^\ast-1)\leq 5$. Furthermore, for $\alpha^\ast\geq 1/6, \alpha^\ast<1$ the matrix $C(\alpha^\ast,2\alpha^\ast-1)$ is positive semidefinite. The matrix $C(1/6,-2/3)$ on the boundary describes the Petersen code \cite{BACHOC}, which corresponds to the midpoints of the regular simplex in $\mathbb{R}^4$. \qed
\end{example}
However, computing the spectrum of a candidate Gram matrix without any apparent structure is a delicate task, and instead we will rely on the following key technical result.
\begin{proposition}[Strong Gr\"obner test, cf.~Corollary~\ref{WGB}]\label{PKEY}
Let $d\geq 2$ be fixed, and let $C(\alpha,\beta)$ be a candidate Gram matrix of order $n\geq d+1$. Let $\mathcal{M}$ denote the set of all $(d+1)\times(d+1)$ submatrices of $C$. Let $\omega$ be an auxiliary variable. If the following system of polynomial equations
\begin{equation}\label{KEYEQ}
\begin{cases}
\det M(\alpha,\beta)=0,\qquad \text{for all $M\in\mathcal{M}$}\\
\omega\alpha\beta(\alpha^2-\beta^2)(\alpha^2-1)(\beta^2-1)+1=0
\end{cases}
\end{equation}
has no solutions in $\mathbb{C}^3$, then $\myrank C(\alpha^\ast,\beta^\ast)\leq d$ cannot hold for any $\alpha^\ast,\beta^\ast\in\mathbb{R}$ subject to \eqref{ASSUME}.
\end{proposition}
\begin{proof}
Indeed, if $\myrank C(\alpha^\ast,\beta^\ast)\leq d$ for some $\alpha^\ast,\beta^\ast\in\mathbb{C}$ subject to \eqref{ASSUME}, then necessarily all $(d+1)\times(d+1)$ minors of $C(\alpha^\ast,\beta^\ast)$ are vanishing. In particular, there exists an $\omega^\ast\in\mathbb{C}$, so that $(\alpha^\ast,\beta^\ast,\omega^\ast)\in\mathbb{C}^3$ is a solution of the system of equations \eqref{KEYEQ}.
\end{proof}
We remark that deciding whether a system of polynomial equations with rational coefficients has any complex solutions can be decided by computing a Gr\"obner basis \cite{BECK}.

Based on these concepts, we now may classify biangular line systems in the following way. First, we fix $d\geq 2$, and $n=\binom{d+3}{4}$. Secondly, we generate (by computers, say) all $n\times n$ candidate Gram matrices. Thirdly, for each candidate Gram matrix $C(\alpha,\beta)$ generated, we attempt to determine, via solving the system of equations \eqref{KEYEQ} the (not necessarily finite) set of all real matrices $\{C(\alpha_i^\ast,\beta_i^\ast)\colon \myrank C(\alpha_i^\ast,\beta_i^\ast)\leq d; i\in\mathcal{I}\}$. Finally, we keep only those which are positive semidefinite. When no such matrices are found, then we decrease $n$ by one and repeat the same procedure.

There are several weak points of this naive method restricting heavily its utility. First of all, the bound on $n$, stipulated by Theorem~\ref{T4} is rather crude, and there is no way to generate all candidate Gram matrices of that size. Secondly, when the solution set of \eqref{KEYEQ} is infinite, then it is a very delicate task to parametrize the matrices $C(\alpha_i^\ast,\beta_i^\ast)$, $i\in\mathcal{I}$, and to describe which of these are positive semidefinite.

We overcome these difficulties by sophisticated matrix generation techniques, and using Proposition~\ref{PKEY} for discarding a large fraction of small candidate Gram matrices. We discuss these efforts in the next subsection.
\subsection{The framework in detail}
In this subsection we describe in more detail how to generate candidate Gram matrices in an equivalence-free exhaustive manner. The main technical tool is canonization, see \cite[Section~4.2.2]{KO}, \cite{REA}. The vectorization of a candidate Gram matrix $C$ of order $n$ is the vector $\myvec(C):=[C_{2,1},C_{3,1},C_{3,2},\dots,C_{n,1},\dots$, $C_{n,n-1}]$. We say that a candidate Gram matrix $C(\alpha,\beta)$ is in canonical form, if it holds that
\begin{multline}\label{CANFORM}
\myvec(C(\alpha,\beta)):=\min\{\myvec(PC(\sigma(\alpha),\sigma(\beta))P^T)\colon \text{$P$ is a signed}\\
\text{permutation matrix, $\sigma$ is a relabeling}\},
\end{multline}
where comparison of vectors is done lexicographically (one may assume, e.g., that the entries are ordered as $0\prec \alpha\prec-\alpha\prec \beta\prec -\beta)$. One particularly attractive feature of the above canonical form is that the leading principal submatrices of canonical matrices are themselves canonical. Therefore canonical matrices can be generated inductively, using smaller canonical matrices as `seeds'. This method is usually called `orderly generation'.
\begin{lemma}
The number of $n\times n$ canonical candidate Gram matrices with entries $\{0,\pm\alpha,\pm\beta\}$ \textup{(}in which all three symbols do not appear simultaneously\textup{)} is given in Table~\ref{TAB1} for $n\in\{1,\dots,8\}$.
\end{lemma}
\begin{table}[htbp]
\caption{The number of candidate Gram matrices up to equivalence}\label{TAB1}
	\begin{tabular}{cccccccccc}
	\hline
	$n$   & 1 & 2 & 3 &  4 &   5 &    6 &       7 & 8\\
	$\#$  & 1 & 2 & 5 & 25 & 194 & 7958 & 1818859 & 1773789830\\
	\hline
	\end{tabular}
\end{table}
\begin{proof}
Case $n=1$ is $\left[\begin{array}{c} 1 \end{array}\right]$, case $n=2$ are $\left[\begin{smallmatrix}1 & 0\\ 0 & 1\end{smallmatrix}\right]$, and $\left[\begin{smallmatrix}1 & \alpha\\ \alpha & 1\end{smallmatrix}\right]$. Case $n=3$ is shown in Example~\ref{EX1}. The remaining cases follow by computation.
\end{proof}
As seen from Table~\ref{TAB1} the number of $n\times n$ candidate Gram matrices grows very rapidly. However, when $d\geq2$ is fixed and $n=d+2$, then we may filter out a very large fraction of candidate Gram matrices with the aid of Proposition~\ref{PKEY}. Indeed, for a given candidate Gram matrix we can check whether \eqref{KEYEQ} has any complex solutions by computing a degree reverse lexicographic reduced Gr\"obner basis \cite{BECK}, and keep only those candidate Gram matrices in a set $\mathcal{C}_d(n)$ for which some solutions are found. We performed this step with the aid of the C++ library `CoCoA' \cite{cocoa}.

We proceed by augmenting each candidate Gram matrix $C\in\mathcal{C}_d(n)$ with a new row (and column) whose prefix $[C_{n+1,1},C_{n+1,2},\dots,C_{n+1,n-1}]$ is lexicographically larger than the respective prefix of the last row of $C$ (cf.~\eqref{CANFORM}), keeping only those canonical matrices which in addition survive the next computationally cheap test.
\begin{lemma}[Combinatorial test]\label{COMBTEST}
Let $d\geq 2$ be fixed, and let $\mathcal{C}_d(n)$ be a set containing all pairwise inequivalent candidate Gram matrices of order $n$ for which the system of equations \eqref{KEYEQ} has a solution. Let $C$ be a candidate Gram matrix of order $n+1$. Then if $C$ corresponds to a Gram matrix in $\mathbb{R}^d$, then necessarily all its $n+1$ principal submatrices of order $n$ belong to the set $\mathcal{C}_d(n)$, up to equivalence.
\end{lemma}
\begin{proof}
Indeed, if $C$ corresponds to some Gram matrix, then there exist real numbers $\alpha^\ast$, $\beta^\ast$ (subject to \eqref{ASSUME}) such that $\myrank C(\alpha^\ast,\beta^\ast)\leq d$. Since the rank of submatrices cannot increase, this must be true for every principal submatrices of $C(\alpha^\ast,\beta^\ast)$. But then these submatrices must be in the set $\mathcal{C}_d(n)$, up to equivalence.
\end{proof}
Since the $n\times n$ principal submatrices of a candidate Gram matrix of order $n+1$ must be compatible, we test them further with the following.
\begin{corollary}[Weak Gr\"obner test, cf.~Proposition~\ref{PKEY}]\label{WGB}
Let $d\geq 2$ be fixed, and let $C(\alpha,\beta)$ be a candidate Gram matrix of order $n\geq d+1$. Let $\mathcal{M}$ denote the set of all $(d+1)\times(d+1)$ principal submatrices of $C$. Let $\omega$ be an auxiliary variable. If the following system of polynomial equations
\begin{equation*}
\begin{cases}
\det M(\alpha,\beta)=0,\qquad \text{for all $M\in\mathcal{M}$}\\
\omega\alpha\beta(\alpha^2-\beta^2)(\alpha^2-1)(\beta^2-1)+1=0
\end{cases}
\end{equation*}
has no solutions in $\mathbb{C}^3$, then $\myrank C(\alpha^\ast,\beta^\ast)\leq d$ cannot hold for any $\alpha^\ast,\beta^\ast\in\mathbb{R}$ subject to \eqref{ASSUME}.
\end{corollary}
\begin{proof}
This is a variant of Proposition~\ref{PKEY}.
\end{proof}
Finally, we store all surviving matrices in a set $\mathcal{C}_d(n+1)$, and repeat this procedure as long as new matrices are discovered (but until $n$ achieves the Absolute bound in Theorem~\ref{T4}). Once the largest candidate Gram matrices are found, we use Proposition~\ref{PKEY} to determine explicitly the matrices with rank at most $d$, and then by computing their characteristic polynomial (or eigenvalues, if it is possible) we determine the positive semidefinite matrices. We remark that the set of inner products of the maximum Gram matrices is a by-product of this procedure.

We summarize our approach in the following `roadmap' which we will frequently use as a convenient reference.
\begin{roadmap}\label{ROADMAP}
The following is our approach for generating and classifying biangular lines in $\mathbb{R}^d$.
\begin{itemize}[leftmargin=*]
\item Fix the dimension $d\geq 2$.
\item Generate all $\{0,\pm\alpha,\pm\beta\}$ canonical candidate Gram matrices \textup{(}with at most two symbols\textup{)} of size $d+1$, and store them in a set $\mathcal{C}_d(d+1)$.
\item Augment every $C\in\mathcal{C}_d(d+1)$ with a new row and column in every possible way, and then test the canonical matrices by Proposition~\ref{PKEY}. Store the surviving matrices of size $d+2$ in a set $\mathcal{C}_d(d+2)$.
\item For every $i\in\{d+2,\dots,\binom{d+3}{4}\}$ augment every $C\in\mathcal{C}_d(i)$ with a new row and column in every possible way, and then test the canonical matrices by Lemma~\ref{COMBTEST} and Corollary~\ref{WGB}. Store the surviving matrices of size $i+1$ in a set $\mathcal{C}_d(i+1)$, and repeat this step.
\item For the largest candidate Gram matrices use Proposition~\ref{PKEY} and in particular the solutions of the system of equations \eqref{KEYEQ} to determine the real matrices of rank at most $d$.
\item Select from these the positive semidefinite matrices.
\end{itemize}
\end{roadmap}
\begin{remark}
We observed that once the size $n$ of candidate Gram matrices is large enough, say $n\geq d+5$, then essentially all matrices survive Corollary~\ref{WGB}. In these cases we solely rely on Lemma~\ref{COMBTEST} for pruning. We believe that the reason for this phenomenon is related to the fact that the congruence order of $\mathbb{R}^d$ is $d+3$, see \cite[Theorem~7.2]{L}.
\end{remark}
\begin{remark}
Let $d\geq 3$, $n\geq d+1$, $\alpha^\ast,\beta^\ast\in\mathbb{R}$ fixed, and let $C(\alpha^\ast,\beta^\ast)$ be an $n\times n$ Gram matrix with $\myrank C(\alpha^\ast,\beta^\ast)\leq d-2$. Then for every $v\in\mathbb{R}^n$, $\myrank \left[\begin{smallmatrix} C(\alpha^\ast,\beta^\ast) & v^T\\ v & 1\end{smallmatrix}\right]\leq d$ by subadditivity. In particular, the  tests described in Proposition~\ref{PKEY} and Corollary~\ref{WGB} have no effect.
\end{remark}
\begin{remark}
There are two major techniques for matrix canonization: the one relies on formula \eqref{CANFORM} which nicely fits into the framework of `orderly generation'. The other possibility is to transform the problem of matrix canonization to graph canonization for which there are readily available efficient implementations, such as the `nauty' software \cite{MCK}. In Appendix~\ref{APPENDIX1} we describe a graph representation of candidate Gram matrices, which can be used in the framework of `canonical augmentation'. These two techniques are of similar efficiency, and we have used both of them to cross-check our results. We refer the reader to \cite{BRINK} and the references therein.
\end{remark}
\section{Classification of maximum biangular lines}\label{sec4}
We implemented the framework developed in Section~\ref{sec3} in C++ and used a computer cluster with 500 CPU cores for several weeks to obtain the following new classification results in $\mathbb{R}^d$ for $d\leq 6$.

For completeness, we begin our discussion with the case $d=2$ by giving an independent, computational proof to Lemma~\ref{TENGONLEMMA}.
\begin{lemma}[Equivalent restatement of Lemma~\ref{TENGONLEMMA}]
The maximum cardinality of a biangular line system in $\mathbb{R}^2$ is $5$. The unique configuration has candidate Gram matrix
\begin{equation}\label{TENG}
C(\alpha,\beta)=\left[
\begin{smallmatrix}
1 & \alpha & \alpha & \beta & \beta \\
\alpha & 1 & \beta & \alpha & \beta \\
\alpha & \beta & 1 & \beta & \alpha \\
\beta & \alpha & \beta & 1 & \alpha \\
\beta & \beta & \alpha & \alpha & 1 \\
\end{smallmatrix}
\right]
\end{equation}
and Gram matrix $C((\sqrt{5}-1)/4,(-\sqrt{5}-1)/4)$, describing the main diagonals of the convex regular $10$-gon.
\end{lemma}
\begin{table}[htbp]%
	\centering
	\caption{$\{0,\pm \alpha,\pm \beta\}$ candidate Gram matrices in $\mathbb{R}^2$}\label{TAB2D}
	\begin{tabular}{c|ccccc}
		\toprule
		$n$                  & 2 & 3 & 4 & 5 & 6\\
	   	$|\mathcal{C}_2(n)|$ & 2 & 3 & 2 & 1 & 0\\
		\bottomrule
	\end{tabular}
\end{table}
\begin{proof}
The proof follows Roadmap~\ref{ROADMAP} with $d=2$. In Table~\ref{TAB2D} we display the number of surviving candidate Gram matrices, that is the numbers $|\mathcal{C}_2(n)|$ for $n\in\{2,\dots,6\}$. Since $|\mathcal{C}_2(6)|=0$, it follows that $|\mathcal{C}_2(n)|=0$ for every $n\geq 6$. The unique maximum candidate Gram matrix of size $5$ is shown in \eqref{TENG} from which the Gram matrices can be recovered by solving the system of equations \eqref{KEYEQ}. It follows that $4\alpha^2+2\alpha-1=0$, and $\beta=-\alpha-1/2$. This yields two permutation equivalent, positive semidefinite solutions: $C((\sqrt{5}-1)/4,(-\sqrt5-1)/4)$ and $C((-\sqrt{5}-1)/4,(\sqrt5-1)/4)$, both corresponding to the main diagonals of the convex regular $10$-gon.
\end{proof}
\begin{remark}
The four lines, passing through the antipodal vertices of the convex regular octagon form the second largest, inextendible configuration of biangular lines in $\mathbb{R}^2$ with set of inner products $\{0,\pm1/\sqrt{2}\}$.
\end{remark}
\begin{theorem}\label{TDOD}
The maximum cardinality of a biangular line system in $\mathbb{R}^3$ is $10$. The unique configuration has candidate Gram matrix
\begin{equation}\label{DOD}
C(\alpha,\beta)=\left[
\begin{smallmatrix}
1 & \alpha & \alpha & \alpha & \alpha & \alpha & \alpha & \beta & \beta & \beta \\
\alpha & 1 & \alpha & -\alpha & -\alpha & \beta & -\beta & \alpha & -\alpha & \beta \\
\alpha & \alpha & 1 & \beta & -\beta & -\alpha & -\alpha & -\alpha & \alpha & \beta \\
\alpha & -\alpha & \beta & 1 & -\alpha & -\beta & \alpha & -\alpha & \beta & \alpha \\
\alpha & -\alpha & -\beta & -\alpha & 1 & \alpha & \beta & \beta & \alpha & -\alpha \\
\alpha & \beta & -\alpha & -\beta & \alpha & 1 & -\alpha & \beta & -\alpha & \alpha \\
\alpha & -\beta & -\alpha & \alpha & \beta & -\alpha & 1 & \alpha & \beta & -\alpha \\
\beta & \alpha & -\alpha & -\alpha & \beta & \beta & \alpha & 1 & \alpha & \alpha \\
\beta & -\alpha & \alpha & \beta & \alpha & -\alpha & \beta & \alpha & 1 & \alpha \\
\beta & \beta & \beta & \alpha & -\alpha & \alpha & -\alpha & \alpha & \alpha & 1 \\
\end{smallmatrix}
\right]
\end{equation}
and Gram matrix $C(1/3,\sqrt5/3)$, corresponding to the main diagonals of the platonic dodecahedron.
\end{theorem}
\begin{table}[htbp]%
	\centering
	\caption{$\{0,\pm \alpha,\pm \beta\}$ candidate Gram matrices in $\mathbb{R}^3$}\label{TAB3D}
	\begin{tabular}{c|cccccccccc}
		\toprule
		$n$ & 2 & 3 &  4 &  5 &  6 & 7 & 8 & 9 & 10 & 11\\
		$|\mathcal{C}_3(n)|$ & 2 & 5 & 22 & 23 & 12 & 5 & 2 & 1 &  1 &  0\\
		\bottomrule
	\end{tabular}
\end{table}
\begin{proof}
The proof follows Roadmap~\ref{ROADMAP} with $d=3$. In Table~\ref{TAB3D} we display the number of surviving candidate Gram matrices, that is the numbers $|\mathcal{C}_3(n)|$ for $n\in\{2,\dots,11\}$. Since $|\mathcal{C}_3(11)|=0$, it follows that $|\mathcal{C}_3(n)|=0$ for every $n\geq 11$. The unique maximum candidate Gram matrix of size $10$ is shown in \eqref{DOD}. The equations \eqref{KEYEQ} imply that $\alpha=1/3$, and $\beta^2=5/9$. This yields two permutation equivalent, positive semidefinite solutions: $C(1/3,\sqrt{5}/3)$ and $C(1/3,-\sqrt5/3)$, both corresponding to the main diagonals of the platonic dodecahedron.
\end{proof}
\begin{remark}
The second largest (inextendible) examples in $\mathbb{R}^3$ can be obtained by lifting the convex regular $7$-gon by Proposition~\ref{P2} to two carefully chosen heights.
\end{remark}
\begin{theorem}\label{THMDIM4}
The maximum cardinality of a biangular line system in $\mathbb{R}^4$ is $12$. There are four pairwise nonisometric maximum configurations: the shortest vectors of the $D_4$ lattice; the shortest vectors of the $D_3$ lattice after lifting; and two spherical $3$-distance sets with common candidate Gram matrix
\begin{equation}\label{eq111}
C(\alpha,\beta)=\left[\begin{smallmatrix}
B(\alpha,\beta)+I_6 & B(\beta,\alpha)-\beta I_6\\
B(\beta,\alpha)-\beta I_6 & B(\alpha,\beta)+I_6\\
\end{smallmatrix}\right],\quad \text{ where }\quad
B(\alpha,\beta)=\left[\begin{smallmatrix}
0 & \alpha  & \alpha  & \alpha  & \alpha  & \alpha\\
\alpha & 0 & \alpha & \beta & \beta & \alpha\\
\alpha & \alpha & 0 & \alpha & \beta & \beta\\
\alpha & \beta & \alpha & 0 & \alpha  & \beta\\
\alpha & \beta & \beta & \alpha  & 0 & \alpha\\
\alpha & \alpha & \beta & \beta & \alpha & 0\\
\end{smallmatrix}\right],
\end{equation}
yielding nonisometric Gram matrices $C((3-2\sqrt{5})/11,(4+\sqrt5)/11)$ and $C((3+2\sqrt{5})/11,(4-\sqrt5)/11)$.
\end{theorem}
\begin{table}[htbp]%
	\centering
	\caption{$\{0,\pm \alpha,\pm \beta\}$ candidate Gram matrices in $\mathbb{R}^4$}\label{TAB4D}
	\begin{tabular}{c|cccccccccccc}
		\toprule
		$n$ & 2 & 3 &  4 &   5 &   6 &   7 &  8 &  9 & 10 & 11 & 12 & 13\\
		$|\mathcal{C}_4(n)|$ & 2 & 5 & 25 & 191 & 701 & 184 & 69 & 27 & 14 &  3 &  3 &  0\\  
		\bottomrule
	\end{tabular}
\end{table}
\begin{proof}
The proof follows Roadmap~\ref{ROADMAP} with $d=4$. In Table~\ref{TAB4D} we display the number of surviving candidate Gram matrices, that is the numbers $|\mathcal{C}_4(n)|$ for $n\in\{2,\dots,13\}$. Since $|\mathcal{C}_4(13)|=0$, it follows that $|\mathcal{C}_4(n)|=0$ for every $n\geq 13$. The candidate Gram matrices corresponding to the $D_4$ and the lifted $D_3$ lattice vectors are not shown here, as they can be easily recovered from Lemma~\ref{L1} and Proposition~\ref{P2}, and one may check by solving \eqref{KEYEQ} that these are the only solutions. Interestingly, the third candidate Gram matrix $C(\alpha,\beta)$ shown in \eqref{eq111} yields two nonisometric solutions, as the equations \eqref{KEYEQ} imply that $11\alpha^2-6\alpha-1=0$, and $\beta=\alpha/2-1/2$.
\end{proof}
We note that since the candidate Gram matrix \eqref{eq111} describes a spherical $3$-distance set, it has already been generated earlier in \cite{SZO}.
\begin{remark}
The Gram matrices obtained from \eqref{eq111} are contained in the Bose--Mesner algebra of a $3$-class association scheme \cite{hanaki}.
\end{remark}
\begin{theorem}\label{THMDIM5}
The maximum cardinality of a biangular line system in $\mathbb{R}^5$ is $24$. The unique configuration can be obtained by lifting the shortest vectors of the $D_4$ lattice.
\end{theorem}
\begin{table}[htbp]%
\centering
\caption{$\{0,\pm \alpha,\pm \beta\}$ candidate Gram matrices in $\mathbb{R}^5$}\label{TAB5D}
\begin{tabular}{lr|lr|lr|lr|lr}
\toprule
$n$ & $|\mathcal{C}_5(n)|$ & $n$ & $|\mathcal{C}_5(n)|$ & $n$ & $|\mathcal{C}_5(n)|$ & $n$ & $|\mathcal{C}_5(n)|$ & $n$ & $|\mathcal{C}_5(n)|$\\
\midrule
6 & 7954 & 10 & 48448 & 14 & 38826 & 18 & 984 & 22 & 4 \\
7 & 47418 & 11 & 54750 & 15 & 22887 & 19 & 201 & 23 & 1 \\
8 & 27905 & 12 & 56548 & 16 & 10533 & 20 & 45 & 24 & 1 \\
9 & 37381 & 13 & 52246 & 17 & 3701 & 21 & 10 & 25 & 0 \\
\bottomrule
\end{tabular}
\end{table}
\begin{proof}
The proof follows Roadmap~\ref{ROADMAP} with $d=5$. In Table~\ref{TAB5D} we display the number of surviving candidate Gram matrices, that is the numbers $|\mathcal{C}_5(n)|$ for $n\in\{6,\dots,25\}$. Since $|\mathcal{C}_5(25)|=0$, it follows that $|\mathcal{C}_5(n)|=0$ for every $n\geq 25$. The candidate Gram matrix corresponding to the lifted $D_4$ lattice vectors is not shown here, as it can be easily recovered from Lemma~\ref{L1} and Proposition~\ref{P2}, and one may check by solving \eqref{KEYEQ} that it is the only maximum solution.
\end{proof}
\begin{remark}
We remark that the Bose--Mesner algebra (see \cite{hanaki}) of a particular example of $4$-class association schemes on $24$ vertices contains the maximum Gram matrix $G$ of biangular lines in $\mathbb{R}^5$, up to equivalence. Furthermore, since $G^2=24/5G$, $G$ is a sporadic example of biangular tight frames \cite{BTF}.
\end{remark}
The main computational result of this paper is the following.
\begin{theorem}\label{THMDIM6}
The maximum cardinality of a biangular line system in $\mathbb{R}^6$ is $40$. The unique configuration can be obtained by lifting the shortest vectors of the $D_5$ lattice.
\end{theorem}
\begin{table}[htbp]%
	\centering
	\caption{$\{0,\pm \alpha,\pm \beta\}$ candidate Gram matrices in $\mathbb{R}^6$}\label{TAB6D}\scriptsize
	\begin{tabular}{lr|lr|lr|lr|lr}
		\toprule
$n$ & $|\mathcal{C}_6(n)|$ & $n$ & $|\mathcal{C}_6(n)|$ & $n$ & $|\mathcal{C}_6(n)|$ & $n$ & $|\mathcal{C}_6(n)|$ & $n$ & $|\mathcal{C}_6(n)|$\\
\midrule
  &   & 14 & 8000713 & 21 & 34995847 & 28 & 1535902 & 35 & 363 \\
8 & 6883459 & 15 & 11810513 & 22 & 30226589 & 29 & 646252 & 36 & 85 \\
9 & 3170550 & 16 & 17409677 & 23 & 23679948 & 30 & 243144 & 37 & 18 \\
10 & 4107292 & 17 & 24048177 & 24 & 16808810 & 31 & 81562 & 38 & 5 \\
11 & 5260036 & 18 & 30449143 & 25 & 10794327 & 32 & 24461 & 39 & 1 \\
12 & 5781148 & 19 & 35103515 & 26 & 6260018 & 33 & 6554 & 40 & 1 \\
13 & 6239734 & 20 & 36779026 & 27 & 3270750 & 34 & 1610 & 41 & 0 \\
		\bottomrule
	\end{tabular}
\end{table}\normalsize
\begin{proof}
The proof follows Roadmap~\ref{ROADMAP} with $d=6$. In Table~\ref{TAB6D} we display the number of surviving candidate Gram matrices, that is the numbers $|\mathcal{C}_6(n)|$ for $n\in\{8,\dots,41\}$. Since $|\mathcal{C}_6(41)|=0$, it follows that $|\mathcal{C}_6(n)|=0$ for every $n\geq 41$. The candidate Gram matrix corresponding to the lifted $D_5$ lattice vectors is not shown here, as it can be easily recovered from Lemma~\ref{L1} and Proposition~\ref{P2}, and one may check by solving \eqref{KEYEQ} that it is the only maximum solution.
\end{proof}
In dimension $5$ and $6$ the largest biangular line systems with irrational angles consist of $20$ and $24$ lines respectively, each having the very same inner product set $\{\pm(3-2\sqrt5)/11,\pm(4+\sqrt5)/11\}$ as one of the largest configurations in $\mathbb{R}^4$ (cf.~Theorem~\ref{THMDIM4}). Examples of these are shown in Appendix~\ref{APPENDIXB}.
\begin{remark}
In $\mathbb{R}^6$ two $27\times 27$ candidate Gram matrices were found corresponding to Gram matrices with angle set $\{\pm1/4,\pm1/2\}$. It turns out, that one of these is the largest spherical $2$-distance set \cite{L}, \cite{NM}, and the other one belongs to the Bose--Mesner algebra of a $4$-class association scheme \cite{hanaki}. See Appendix~\ref{APPENDIXB}.
\end{remark}
We conclude this section with the following by-products of our classification.
\begin{corollary}\label{CABOVE}
The largest infinite family of biangular lines in $\mathbb{R}^d$ for $d\in\{3,4,5,6\}$ is formed by $6$, $6$, $10$, and $16$ lines, respectively.
\end{corollary}
\begin{proof}
For $d=3$ we have the twisted icosahedron \cite{BTF}. For $d\geq 4$, we can use Proposition~\ref{PINF} and well-known spherical $2$-distance sets (see \cite{L}, \cite{NM}, Example~\ref{EXPET} and Example~\ref{EX3}) in $\mathbb{R}^{d-1}$ to establish the claimed lower bounds. To see that these are indeed the largest, one should inspect the candidate Gram matrices we generated. It is easy to see that if $C(\alpha,\beta)$ is a parametric family of biangular line systems, then so is every subsytem of it. Therefore it is enough to augment those (rather few) candiate Gram matrices for which the dimension of the ideal, generated by \eqref{KEYEQ} is positive (see \cite{BECK}).
\end{proof}
\begin{corollary}\label{RELBDCOR}
The biangular line systems meeting the relative bound in dimension $d\in\{3,4,5,6\}$ for $\alpha^2+\beta^2<6/(d+4)$ are exactly those listed in Table~\ref{TABREL}.
\end{corollary}
\begin{proof}
Let $\mathcal{X}\subset\mathbb{R}^d$ span a biangular line system meeting the relative bound \eqref{RELBD}. Since $\alpha^2+\beta^2<6/(d+4)$, we have $\sum_{x,x'\in\mathcal{X}}C_2^{((d-2)/2)}(\left\langle x,x'\right\rangle)=0$ and $\sum_{x,x'\in\mathcal{X}}C_4^{((d-2)/2)}(\left\langle x,x'\right\rangle)=0$. In particular, the antipodal double $\mathcal{Y}:=\{x\colon x\in\mathcal{X}\}\cup\{-x\colon x\in\mathcal{X}\}$ is a spherical $5$-design \cite{BAN9}, \cite{BOY}, and hence $|\mathcal{X}|=|\mathcal{Y}|/2\geq d(d+1)/2$. For $d=3$ the only tight spherical $5$-design is the icosahedron \cite{BAN9}, \cite[Example~5.16]{DGS2}. For $d\geq 4$ it follows from Corollary~\ref{CABOVE} that the number of Gram matrices of size $|\mathcal{X}|$ is finite, therefore one may plug in the (finitely many) inner products $\alpha^\ast$ and $\beta^\ast$ into \eqref{RELBD} to test equality. This yields Table~\ref{TABREL} for $d\leq 6$.
\end{proof}
\begin{remark}\label{RELBDRM}
If $\alpha^2+\beta^2=6/(d+4)$ and there is equality in the relative bound \eqref{RELBD}, then necessarily $\frac{d^2+3|\mathcal{X}|-4}{(d+2)(d+4)}=(|\mathcal{X}|-1)\beta^2(\frac{6}{d+4}-\beta^2)$. For fixed $d$ and $|\mathcal{X}|$ this in turn determines the possible inner products in $A(\mathcal{X})$. Then one may go through all candidate Gram matrices and check which of these inner products are compatible with the solutions of \eqref{KEYEQ}. Since we tend to believe that for $d\leq 6$ there are no biangular lines of this type, we have not gone through the details of this lengthy and seemingly very tedious task.
\end{remark}
%%%%%%%%%%%%%%%%%%%%%%%%%%%%%%%%%%%%%%%%%%%%
% SECTION 5: Results on multiangular lines%
%%%%%%%%%%%%%%%%%%%%%%%%%%%%%%%%%%%%%%%%%%%%
\section{Results on multiangular lines}\label{sec5}
The theory developed in Section~\ref{sec3} can be generalized to multiangular lines in a straightforward manner. The main challenge in our study is solving (the multiangular analogue of) the system of equations \eqref{KEYEQ}. Indeed, the efficiency of computing a Gr\"obner basis very much depends on the number of variables \cite{BECK}, and $4$-angular line systems are the largest ones our methods can currently handle. In this section we briefly report on our computational results on multiangular lines.
\subsection{Multiangular lines in $\mathbb{R}^3$}
It is well-known that in $\mathbb{R}^3$ the main diagonals of the platonic icosahedron forms the largest equiangular line system, and we showed in Theorem~\ref{TDOD} that the main diagonals of the platonic dodecahedron forms the largest biangular line system. It is natural to ask what are the multiangular analogues of these objects.

It is well-known that on the plane the maximum cardinality of $m$-angular lines is $2m+1$, and an example is coming from the main diagonals of the convex regular $(4m+2)$-gon \cite{NM}.
\begin{theorem}\label{MULTHM}
The maximum cardinality of a triangular line system in $\mathbb{R}^3$ is $12$. There are exactly two such configurations coming from the following candidate Gram matrix\textup{:}
\begin{equation}\label{TRI1}
C(\alpha,\beta,\gamma)=\left[\begin{smallmatrix}
1 & \alpha & \alpha & \alpha & \alpha & \beta & \beta & \beta & \beta & \gamma & \gamma & \gamma \\
\alpha & 1 & \beta & \gamma & \gamma & \alpha & \beta & \beta & \gamma & \alpha & \alpha & \beta \\
\alpha & \beta & 1 & \gamma & -\alpha & \gamma & -\beta & -\gamma & \alpha & \beta & -\beta & \alpha \\
\alpha & \gamma & \gamma & 1 & \beta & \beta & \alpha & \gamma & \beta & \alpha & \beta & \alpha \\
\alpha & \gamma & -\alpha & \beta & 1 & -\beta & \gamma & \alpha & -\gamma & \beta & \alpha & -\beta \\
\beta & \alpha & \gamma & \beta & -\beta & 1 & -\gamma & -\beta & \alpha & \gamma & -\alpha & \alpha \\
\beta & \beta & -\beta & \alpha & \gamma & -\gamma & 1 & \alpha & -\beta & \gamma & \alpha & -\alpha \\
\beta & \beta & -\gamma & \gamma & \alpha & -\beta & \alpha & 1 & -\alpha & \alpha & \gamma & -\beta \\
\beta & \gamma & \alpha & \beta & -\gamma & \alpha & -\beta & -\alpha & 1 & \alpha & -\beta & \gamma \\
\gamma & \alpha & \beta & \alpha & \beta & \gamma & \gamma & \alpha & \alpha & 1 & \beta & \beta \\
\gamma & \alpha & -\beta & \beta & \alpha & -\alpha & \alpha & \gamma & -\beta & \beta & 1 & -\gamma \\
\gamma & \beta & \alpha & \alpha & -\beta & \alpha & -\alpha & -\beta & \gamma & \beta & -\gamma & 1 \\
\end{smallmatrix}\right],
\end{equation}
namely $C((-7+4\sqrt2)/17,(5+2\sqrt2)/17,(-3-8\sqrt2)/17)$ is the truncated cube, and $C((-7-4\sqrt2)/17,(5-2\sqrt2)/17,(-3+8\sqrt2)/17)$ is the small rhombicuboctahedron.
\end{theorem}
\begin{table}[htbp]%
	\centering
	\caption{$\{0,\pm \alpha,\pm \beta, \pm\gamma\}$ candidate Gram matrices in $\mathbb{R}^3$}\label{TABMULT1}
	\begin{tabular}{c|cccccccccccc}
	\toprule
	$n$ & 2 & 3 &  4 &  5 &  6 & 7 & 8 & 9 & 10 & 11 & 12 & 13\\
	$|\mathcal{C}_3(n)|$ & 2 & 7 & 62 & 610 & 271 & 104 & 46 & 19 &  6 &  1 & 1 & 0\\
	\bottomrule
\end{tabular}
\end{table}
\begin{proof}
The proof follows analogously to Roadmap~\ref{ROADMAP} with $d=3$. In Table~\ref{TABMULT1} we display the number of surviving candidate Gram matrices with symbols $\{0,\pm\alpha,\pm\beta$, $\pm\gamma\}$ (where at most three out of these four symbols appear), that is the numbers $|\mathcal{C}_3(n)|$ for $n\in\{2,\dots,13\}$. Since $|\mathcal{C}_3(13)|=0$, it follows that $|\mathcal{C}_3(n)|=0$ for every $n\geq 13$. In addition, there is a unique maximum candidate Gram matrix of size $12$, as shown in \eqref{TRI1}. Analogous equations to \eqref{KEYEQ} imply the claimed solutions.
\end{proof}
\begin{theorem}\label{IDOD}
The maximum cardinality of a $4$-angular line system in $\mathbb{R}^3$ is $15$. There is a unique configuration coming from the following candidate Gram matrix\textup{:}
\begin{equation}\label{icosid}
C(\alpha,\beta,\gamma)=\left[\begin{smallmatrix}1 & 0 & 0 & \alpha & \beta & \gamma & \alpha & \beta & \gamma & \alpha & \beta & \gamma & \alpha & \beta & \gamma \\
0 & 1 & 0 & \beta & \gamma & \alpha & \beta & \gamma & \alpha & -\beta & -\gamma & -\alpha & -\beta & -\gamma & -\alpha \\
0 & 0 & 1 & \gamma & \alpha & \beta & -\gamma & -\alpha & -\beta & \gamma & \alpha & \beta & -\gamma & -\alpha & -\beta \\
\alpha & \beta & \gamma & 1 & 0 & 0 & \gamma & -\alpha & -\beta & \alpha & -\beta & \gamma & -\beta & -\gamma & \alpha \\
\beta & \gamma & \alpha & 0 & 1 & 0 & -\alpha & \beta & \gamma & -\beta & \gamma & -\alpha & -\gamma & -\alpha & \beta \\
\gamma & \alpha & \beta & 0 & 0 & 1 & -\beta & \gamma & \alpha & \gamma & -\alpha & \beta & \alpha & \beta & -\gamma \\
\alpha & \beta & -\gamma & \gamma & -\alpha & -\beta & 1 & 0 & 0 & -\beta & -\gamma & \alpha & \alpha & -\beta & \gamma \\
\beta & \gamma & -\alpha & -\alpha & \beta & \gamma & 0 & 1 & 0 & -\gamma & -\alpha & \beta & -\beta & \gamma & -\alpha \\
\gamma & \alpha & -\beta & -\beta & \gamma & \alpha & 0 & 0 & 1 & \alpha & \beta & -\gamma & \gamma & -\alpha & \beta \\
\alpha & -\beta & \gamma & \alpha & -\beta & \gamma & -\beta & -\gamma & \alpha & 1 & 0 & 0 & \gamma & -\alpha & -\beta \\
\beta & -\gamma & \alpha & -\beta & \gamma & -\alpha & -\gamma & -\alpha & \beta & 0 & 1 & 0 & -\alpha & \beta & \gamma \\
\gamma & -\alpha & \beta & \gamma & -\alpha & \beta & \alpha & \beta & -\gamma & 0 & 0 & 1 & -\beta & \gamma & \alpha \\
\alpha & -\beta & -\gamma & -\beta & -\gamma & \alpha & \alpha & -\beta & \gamma & \gamma & -\alpha & -\beta & 1 & 0 & 0 \\
\beta & -\gamma & -\alpha & -\gamma & -\alpha & \beta & -\beta & \gamma & -\alpha & -\alpha & \beta & \gamma & 0 & 1 & 0 \\
\gamma & -\alpha & -\beta & \alpha & \beta & -\gamma & \gamma & -\alpha & \beta & -\beta & \gamma & \alpha & 0 & 0 & 1 \\
\end{smallmatrix}\right],
\end{equation}
namely $C((1+\sqrt5)/4,(1-\sqrt5)/4,1/2)$ is the icosidodecahedron.
\end{theorem}
\begin{table}[htbp]%
	\centering
	\caption{$\{0,\pm \alpha,\pm \beta, \pm\gamma,\pm\delta\}$ candidate Gram matrices in $\mathbb{R}^3$}\label{TABMULT2}
	\begin{tabular}{lr|lr|lr|lr|lr}
		\toprule
		$n$ & $|\mathcal{C}_3(n)|$ & $n$ & $|\mathcal{C}_3(n)|$ & $n$ & $|\mathcal{C}_3(n)|$ & $n$ & $|\mathcal{C}_3(n)|$ & $n$ & $|\mathcal{C}_3(n)|$\\
		\midrule
		2 & 2  & 5 & 7014 &  8 & 632 & 11 & 32 & 14 & 1\\ 
		3 & 7  & 6 & 7744 &  9 & 276 & 12 & 14 & 15 & 1\\
		4 & 97 & 7 & 1655 & 10 & 104 & 13 &  3 & 16 & 0\\
		\bottomrule
	\end{tabular}
\end{table}
\begin{proof}
The proof follows analogously to Roadmap~\ref{ROADMAP} with $d=3$, with the following noted difference: first we generated all $5\times 5$ candidate Gram matrices, and used Proposition~\ref{PKEY} for filtering the $6\times 6$ (and larger) matrices. In Table~\ref{TABMULT2} we display the number of surviving candidate Gram matrices with symbols $\{0,\pm\alpha,\pm\beta,\pm\gamma,\pm\delta\}$, (where at most four out of these five symbols appear), that is, the numbers $|\mathcal{C}_3(n)|$ for $n\in\{2,\dots,16\}$. Since $|\mathcal{C}_3(16)|=0$, it follows that $|\mathcal{C}_3(n)|=0$ for every $n\geq 16$. In addition, there is a unique maximum candidate Gram matrix of size $15$, as shown in \eqref{icosid}. Analogous equations to \eqref{KEYEQ} imply that $4\alpha^2-2\alpha-1=0$, $\beta=1/2-\alpha$, $\gamma=1/2$. This yields two equivalent, positive semidefinite solutions, both corresponding to the main diagonals of the icosidodecahedron.
\end{proof}
\begin{remark}
It turns out, that the icosidodecahedron is the largest $5$-angular configuration in $\mathbb{R}^3$ containing orthogonal lines. The search is completely analogous to what is described in Theorem~\ref{IDOD} and its proof.
\end{remark}
We refer the reader to \cite{SLOANE2} for further interesting arrangements in $\mathbb{R}^3$.
\subsection{Higher dimensional examples}
In this section we report on our computational results on triangular line systems, where one of the three possible inner products is $0$. On the plane, the unique maximum configuration is formed by the main diagonals of the convex regular $12$-gon, and in dimension $3$ it is once again the main diagonals of the dodecahedron. Both of these results can be concluded from inspecting the matrices what we generated for the proof of Theorem~\ref{MULTHM} (see Table~\ref{TABMULT1}).
\begin{theorem}
The maximum cardinality of a triangular line system containing orthogonal lines in $\mathbb{R}^4$, is $24$. There is a unique configuration spanned by
\begin{multline*}\label{zzz}
\mathcal{X}=\{[1,\pm1,\pm1,\pm1]/2\}\cup\{[1,0,0,0],[0,1,0,0],[0,0,1,0],[0,0,0,1]\}\\
{}\cup\{x\colon \text{$x$ is a permutation of } [\pm1,\pm1,0,0]/\sqrt2; \left\langle x,[4,3,2,1]\right\rangle>0\}%\colon\text{ first nonzero entry is }>0\}
\end{multline*}
which describes the main diagonals of the $24$-cell, and its dual.
\end{theorem}
\begin{table}[htbp]%
	\centering
	\caption{$\{0,\pm \alpha,\pm \beta\}$ candidate Gram matrices in $\mathbb{R}^4$}\label{TABMULT3}
	\begin{tabular}{lr|lr|lr|lr|lr}
		\toprule
		$n$ & $|\mathcal{C}_4(n)|$ & $n$ & $|\mathcal{C}_4(n)|$ & $n$ & $|\mathcal{C}_4(n)|$ & $n$ & $|\mathcal{C}_4(n)|$ & $n$ & $|\mathcal{C}_4(n)|$\\
		\midrule
		1 &    1 &  6 & 8353 & 11 & 2694 & 16 & 892 & 21 & 10\\ 
		2 &    2 &  7 & 2746 & 12 & 2919 & 17 & 447 & 22 &  4\\
		3 &    6 &  8 & 1725 & 13 & 2638 & 18 & 214 & 23 &  1\\
	    4 &   51 &  9 & 1776 & 14 & 2147 & 19 &  80 & 24 &  1\\
		5 & 1152 & 10 & 2314 & 15 & 1453 & 20 &  34 & 25 &  0\\
		\bottomrule
	\end{tabular}
\end{table}
\begin{proof}
The proof follows analogously to Roadmap~\ref{ROADMAP} with $d=4$. In Table~\ref{TABMULT3} we display the number of surviving candidate Gram matrices with symbols $\{0,\pm\alpha,\pm\beta\}$, that is, the numbers $|\mathcal{C}_4(n)|$ for $n\in\{2,\dots,25\}$. Since $|\mathcal{C}_4(25)|=0$, it follows that $|\mathcal{C}_4(n)|=0$ for every $n\geq 25$. The unique largest candidate Gram matrix corresponding to this case can be easily recovered from $\mathcal{X}$, and then solving \eqref{KEYEQ} yields two equivalent solutions with set of inner products $\{0,\pm1/2,\pm 1/\sqrt2\}$.
\end{proof}
\begin{remark}
In $\mathbb{R}^4$, the second largest inextendible configuration has cardinality $16$, spanned by all permutations of $[\pm1,\pm1,\pm1,0]/\sqrt3$ where the first nonzero entry is positive. The set of inner products of this configuration is $\{0,\pm1/3,\pm2/3\}$.
\end{remark}
\begin{theorem}
The maximum cardinality of a triangular line system containing orthogonal lines in $\mathbb{R}^5$ is $40$. This unique configuration is spanned by the set $\mathcal{X}$ of all permutations of $[\pm1,\pm1,\pm1,0,0]/\sqrt3$ where the first nonzero entry is positive.
\end{theorem}
\begin{table}[htbp]%
	\centering
	\caption{$\{0,\pm \alpha,\pm \beta\}$ candidate Gram matrices in $\mathbb{R}^5$}\label{TABMUL4}\scriptsize
	\begin{tabular}{lr|lr|lr|lr|lr}
		\toprule
		$n$ & $|\mathcal{C}_5(n)|$ & $n$ & $|\mathcal{C}_5(n)|$ & $n$ & $|\mathcal{C}_5(n)|$ & $n$ & $|\mathcal{C}_5(n)|$ & $n$ & $|\mathcal{C}_5(n)|$\\
		\midrule
		 7 & 1045395 & 14 & 12214161 & 21 & 68512201 & 28 & 2932142 & 35 & 471  \\
		 8 &  370512 & 15 & 21063583 & 22 & 59177264 & 29 & 1217479 & 36 &  94\\
	   	 9 &  441556 & 16 & 32845898 & 23 & 46323247 & 30 &  449091 & 37 &  18 \\
		10 &  724198 & 17 & 46331977 & 24 & 32824635 & 31 &  146385 & 38 &  4 \\
		11 & 1422041 & 18 & 59180410 & 25 & 21019703 & 32 &   41984 & 39 &  1 \\
		12 & 3076847 & 19 & 68513149 & 26 & 12137301 & 33 &   10565 & 40 &  1 \\
		13 & 6412829 & 20 & 71935169 & 27 &  6301866 & 34 &    2357 & 41 &  0 \\
		\bottomrule
	\end{tabular}
\end{table}\normalsize
\begin{proof}
The proof follows analogously to Roadmap~\ref{ROADMAP} with $d=5$. In Table~\ref{TABMUL4} we display the number of surviving candidate Gram matrices with symbols $\{0,\pm\alpha,\pm\beta\}$, that is, the numbers $|\mathcal{C}_5(n)|$ for $n\in\{7,\dots,41\}$. Since $|\mathcal{C}_5(41)|=0$, it follows that $|\mathcal{C}_5(n)|=0$ for every $n\geq 41$. The unique largest candidate Gram matrix corresponding to this case can be easily recovered from $\mathcal{X}$, and then solving \eqref{KEYEQ} yields a unique solution with set of inner products $\{0,\pm1/3,\pm 2/3\}$.
\end{proof}
%%%%%%%%%%%%%%%%%
% OPEN PROBLEMS %
%%%%%%%%%%%%%%%%%
\section{Open problems}\label{sec6}
We conclude this paper with the following set of problems.
\begin{problem}[Superquadratic lines, see \cite{SUD}]\label{PROB1}
Let $c,\varepsilon>0$ be fixed. Find a construction of a series of biangular lines $\mathcal{X}_d\subset\mathbb{R}^d$, such that $|\mathcal{X}_d|\geq c\cdot d^{2+\varepsilon}$ holds for infinitely many $d\geq 1$.
\end{problem}
In particular, investigate if Proposition~\ref{P2} can be applied to a suitable series of spherical $3$-distance sets.
\begin{problem}
Find a series of spherical $3$-distance sets $\mathcal{X}_d\subset\mathbb{R}^d$ with $A(\mathcal{X}_d)\subseteq\{\alpha_d,\beta_d,\gamma_d\}$ such that $\alpha_d+\beta_d<0$ and $|\mathcal{X}|$ is superquadratic (in the sense of Problem~\ref{PROB1}).
\end{problem}
\begin{problem}[See \cite{BTF}]
Find a series of biangular tight frames $\mathcal{X}_d\subset\mathbb{R}^d$ such that $|\mathcal{X}_d|>d^2$ for infinitely many $d\geq 1$.
\end{problem}
It is known that the twisted icosahedron \cite{BTF} forms an infinite family of $6$ biangular lines in $\mathbb{R}^3$, which is one line larger compared to what Proposition~\ref{PINF} guarantees.
\begin{problem}[Cf.~Corollary~\ref{CABOVE}]
Determine if there exists an infinite family of biangular lines $\mathcal{X}(h)\subset\mathbb{R}^d$ such that $|\mathcal{X}(h)|$ is larger than the one described in Proposition~\ref{PINF} for some $d>6$.
\end{problem}
\begin{problem}[See \cite{LS}, cf.~Example~\ref{EX3}]
Determine if there exist an infinite family of $28$ biangular lines $\mathcal{X}(h)\subset\mathbb{R}^7$ such that $\mathcal{X}(0)$ spans equiangular lines.
\end{problem}
It would be also very interesting to see whether binary codes with four distinct distances lead to improved constructions in $\mathbb{R}^d$ for some $d\leq 23$ or possibly beyond.
\begin{problem}[See Lemma~\ref{LCODE}]
For $d\geq 1$ determine the maximum cardinality of binary codes of length $d$ admitting at most four distinct Hamming distances $\{\Delta_1,\Delta_2,d-\Delta_1,d-\Delta_2\}$, $\Delta_1,\Delta_2\in\{1,\dots,d-1\}$.
\end{problem}
\begin{problem}[Cf.~Theorem~\ref{RELBDS}, Remark~\ref{RELBDRM}]
Determine if there exists a set $\mathcal{X}\subset\mathbb{R}^d$ spanning biangular lines with $A(\mathcal{X})\subseteq\{\pm\alpha,\pm\beta\}$, such that $\alpha^2+\beta^2=6/(d+4)$, and there is equality in \eqref{RELBD} for some $d>1$.
\end{problem}
\section*{Acknowledgements}
We thank Prof.~Patric~\"Osterg\aa rd for providing us (essentially unlimited) access to supercomputing resources at Aalto University.

\appendix
\section{Graph representation of candidate Gram matrices}\label{APPENDIX1}
Let $m\geq 1$, and $n\geq 2$ be integers, and consider an $n\times n$ symmetric matrix $C(\alpha_1,\dots,\alpha_m)$ with constant diagonal entries $1$ over the polynomial ring $\mathbb{Q}[\alpha_1,\dots,\alpha_m]$ with off-diagonal entries $\{0,\pm\alpha_1,\dots$, $\pm\alpha_m\}$. Analogously as set forth earlier in Definition~\ref{MAINDEF}, two such matrices $C_1$ and $C_2$ are called equivalent, if
\begin{equation*}
C_1(\alpha_1,\dots,\alpha_m)=PC_2(\sigma(\alpha_1),\dots,\sigma(\alpha_m))P^T
\end{equation*}
for some signed permutation matrix $P$ and relabeling $\sigma$. A representative of this matrix equivalence class is called a candidate Gram matrix.

The goal of this section is to construct for every matrix $C(\alpha_1,\dots,\alpha_m)$ of order $n$ a (colored) graph $X(C(\alpha_1,\dots,$ $\alpha_m))$ capturing its underlying symmetries and in particular its equivalence class. With this representation, equivalence of matrices $C_1$ and $C_2$ (over the same symbol set) simply boils down to the isomorphism of the corresponding colored graphs $X(C_1)$ and $X(C_2)$. This latter task can be readily decided by the `nauty' software \cite{MCK} in practice.

Our graph $X(C)$ has $2n^2+n+2m$ vertices, and its vertex set $V(X(C))$ is partitioned by the following four distinct (nonempty) color classes:
\begin{equation*}
V(X(C)):=\mathcal{U}\cup\mathcal{V}\cup\mathcal{W}\cup\mathcal{Z}.
\end{equation*}
Here $\mathcal{U}:=\{u_i\colon i\in\{1,\dots,n\}\}$ conceptually represents the $n$ lines (in other words, the $n$ rows/columns of the matrix $C$). The set $\mathcal{V}:=\{v_{ik}\colon i\in\{1,\dots,n\};k\in\{1,2\}\}$ represents the set of antipodal unit vectors (say $\pm x$) spanning the lines. The set $\mathcal{W}:=\{w_{ijk}\colon i<j\in\{1,\dots,n\};k\in\{1,\dots,4\}\}$ represents the four possible inner products $\left\langle \pm x,\pm x'\right\rangle$ (where $\pm x$ and $\pm x'$ are the spanning unit vectors of distinct lines), and finally $\mathcal{Z}=\{z_{ik}\colon i\in\{1,\dots,m\};k\in\{1,2\}\}$ represents the $2m$ off-diagonal entries (where for every $i\in\{1,\dots,m\}$, the vertices $z_{i1}$ and $z_{i2}$ correspond to the same symbol $\alpha_i$ and its negative, in some order).

The edge set, $E(X(C))$ is the following:
\begin{equation*}
E(X(C)):=\mathcal{E}_1\cup\mathcal{E}_2\cup\mathcal{E}_3\cup\mathcal{E}_4\cup\mathcal{E}_5.
\end{equation*}
Here $\mathcal{E}_1:=\{\{u_i,v_{ik}\}\colon i\in\{1,\dots,n\};k\in\{1,2\}\}$ and $\mathcal{E}_2:=\{\{z_{i1},z_{i2}\}\colon i\in\{1,\dots,m\}\}$ describe the edges connecting the elements of $\mathcal{U}$ and $\mathcal{V}$, and the edges within $\mathcal{Z}$, respectively. Furthermore,
\begin{multline*}
\mathcal{E}_3:=\{\{v_{i1},w_{ij1}\},\{v_{i2},w_{ij1}\},\{w_{ij1},w_{ij2}\},\{w_{ij2},w_{ij3}\},\{w_{ij3},w_{ij4}\},\\
\{v_{j1},w_{ij4}\},\{v_{j2},w_{ij4}\}\colon i<j\in\{1,\dots,n\};G_{ij}=0\}
\end{multline*}
and
\begin{multline*}
\mathcal{E}_4:=\{\{v_{ik},w_{ijk}\},\{v_{jk},w_{ijk}\},\{v_{ik},w_{ij(k+2)}\},\{v_{j(3-k)},w_{ij(k+2)}\}\colon\\
i<j\in\{1,\dots,n\};k\in\{1,2\};G_{ij}\neq 0\}
\end{multline*}
describe the graph structure between (vertices representing) orthogonal and non-orthogonal lines, respectively. Finally,
\begin{multline*}
\mathcal{E}_5:=\{\{w_{ijk},z_{\ell1}\},\{w_{ij(k+2)},z_{\ell2}\}\colon i<j\in\{1,\dots,n\};\\
k\in\{1,2\};G_{ij}=\alpha_{\ell}\}\cup\{\{w_{ijk},z_{\ell2}\},\{w_{ij(k+2)},z_{\ell1}\}\colon\\ i<j\in\{1,\dots,n\};k\in\{1,2\};G_{ij}=-\alpha_{\ell}\}
\end{multline*}
describes the edges connecting the vertices between $\mathcal{W}$ and $\mathcal{Z}$, thus providing a correspondence between lines with certain inner products, and the symbols representing these inner products.

The following is a technical statement clarifying the usefulness of such a representation.
\begin{proposition}
The matrices $C_1$ and $C_2$ \textup{(}over the same symbol set\textup{)} are equivalent, if and only if $X(C_1)$ and $X(C_2)$ are isomorphic as graphs. Furthermore, the automorphism groups of $C_1$ and $X(C_1)$ are isomorphic as groups.
\end{proposition}
We omit the proof, and refer the reader to \cite{MCK}. Instead, we show how to represent $\left[\begin{smallmatrix}1 & 0 & \alpha\\ 0 & 1 & \alpha\\ \alpha & \alpha & 1\end{smallmatrix}\right]$ over the symbol set $\{\pm\alpha,\pm\beta\}$ on Figure~\ref{figure1}.
\begin{figure}
	\centering
\begin{tikzpicture}
\node [draw, regular polygon, regular polygon sides=3, inner sep=0pt] (u1) at (-5.000,0) {\tiny$u_{1}$};
\node [draw, regular polygon, regular polygon sides=3, inner sep=0pt] (u2) at (5.000,0) {\tiny$u_{2}$};
\node [draw, regular polygon, regular polygon sides=3, inner sep=0pt] (u3) at (0,-6.180) {\tiny$u_{3}$};
\node [draw, regular polygon, regular polygon sides=4, inner sep=0pt] (v11) at (-4.000,0) {\tiny$v_{11}$};
\node [draw, regular polygon, regular polygon sides=4, inner sep=0pt] (v12) at (-4.500,-0.8660) {\tiny$v_{12}$};
\node [draw, regular polygon, regular polygon sides=4, inner sep=0pt] (v21) at (4.500,-0.8660) {\tiny$v_{21}$};
\node [draw, regular polygon, regular polygon sides=4, inner sep=0pt] (v22) at (4.000,0) {\tiny$v_{22}$};
\node [draw, regular polygon, regular polygon sides=4, inner sep=0pt] (v31) at (-0.5000,-5.314) {\tiny$v_{31}$};
\node [draw, regular polygon, regular polygon sides=4, inner sep=0pt] (v32) at (0.5000,-5.314) {\tiny$v_{32}$};
\node [draw, circle, inner sep=1pt] (w121) at (-1.500,-0.4330) {\tiny$w_{121}$};
\node [draw, circle, inner sep=1pt] (w122) at (-0.5000,-0.4330) {\tiny$w_{122}$};
\node [draw, circle, inner sep=1pt] (w123) at (0.5000,-0.4330) {\tiny$w_{123}$};
\node [draw, circle, inner sep=1pt] (w124) at (1.500,-0.4330) {\tiny$w_{124}$};
\node [draw, circle, inner sep=1pt] (w231) at (4.20721,-2.58615) {\tiny$w_{231}$};
\node [draw, circle, inner sep=1pt] (w232) at (3.25684,-2.27503) {\tiny$w_{232}$};
\node [draw, circle, inner sep=1pt] (w233) at (2.30647,-1.96391) {\tiny$w_{233}$};
\node [draw, circle, inner sep=1pt] (w234) at (1.3561,-1.65279) {\tiny$w_{234}$};
\node [draw, circle, inner sep=1pt] (w131) at (-1.3561,-1.65279) {\tiny$w_{131}$};
\node [draw, circle, inner sep=1pt] (w132) at (-2.30647,-1.96391) {\tiny$w_{132}$};
\node [draw, circle, inner sep=1pt] (w133) at (-3.25684,-2.27503) {\tiny$w_{133}$};
\node [draw, circle, inner sep=1pt] (w134) at (-4.20721,-2.58615) {\tiny$w_{134}$};
\node [draw, diamond, inner sep=1pt] (z11) at (0,-3.3) {\tiny$z_{11}$};
\node [draw, diamond, inner sep=1pt] (z12) at (0,-4.3) {\tiny$z_{12}$};
\node [draw, diamond, inner sep=1pt] (z21) at (0,-1.3) {\tiny$z_{21}$};
\node [draw, diamond, inner sep=1pt] (z22) at (0,-2.3) {\tiny$z_{22}$};
%%%
\draw (u1) -- (v11);
\draw (u1) -- (v12);
\draw (u2) -- (v21);
\draw (u2) -- (v22);
\draw (u3) -- (v31);
\draw (u3) -- (v32);
\draw (v11) -- (w121);
\draw (v12) -- (w121);
\draw (w121) -- (w122);
\draw (w122) -- (w123);
\draw (w123) -- (w124);
\draw (v21) -- (w124);
\draw (v22) -- (w124);
%%%
\draw (v11) -- (w131);
\draw (v11) -- (w133);
\draw (v12) -- (w132);
\draw (v12) -- (w134);
\draw (w131) -- (v31);
\draw (w133) -- (v32);
\draw (w132) -- (v31);
\draw (w134) -- (v32);
%%%
\draw (v21) -- (w231);
\draw (v21) -- (w233);
\draw (v22) -- (w232);
\draw (v22) -- (w234);
\draw (w231) -- (v31);
\draw (w233) -- (v32);
\draw (w232) -- (v31);
\draw (w234) -- (v32);
\draw (z11) -- (z12);
\draw (z21) -- (z22);
\draw (z11) -- (w131);
\draw (z11) -- (w134);
\draw (z12) -- (w132);
\draw (z12) -- (w133);
\draw (z11) -- (w231);
\draw (z11) -- (w234);
\draw (z12) -- (w232);
\draw (z12) -- (w233);
\end{tikzpicture}
	\caption{Graph representation of a candidate Gram matrix}\label{figure1}
\end{figure}
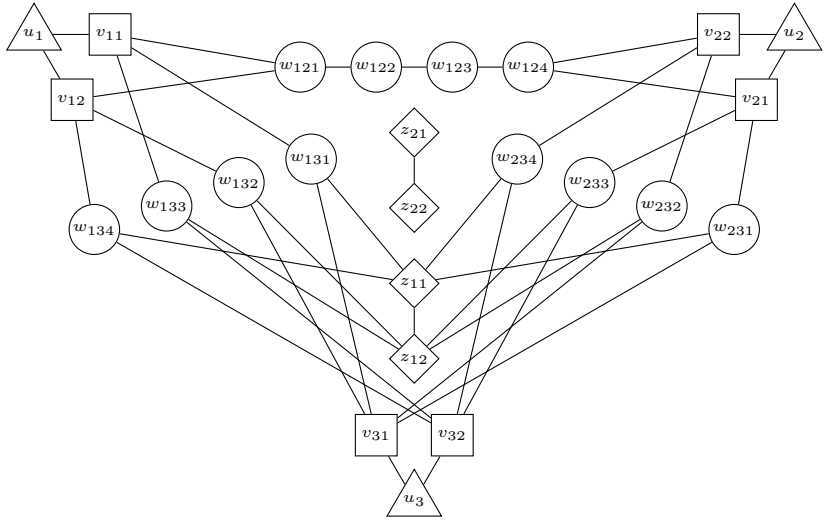
\section{Miscellaneous matrices}\label{APPENDIXB}
We note the largest biangular line systems in $\mathbb{R}^5$ and $\mathbb{R}^6$ containing a pair of lines with irrational inner product between them. It turns out that all of these examples have inner product set $\{\pm\alpha^\ast,\pm\beta^\ast\}$, where $\alpha^\ast:=(3-2\sqrt5)/11$, $\beta^\ast:=(4+\sqrt5)/11$ are the very same values as stated in Theorem~\ref{THMDIM4}. Furthermore, the two examples shown below are extensions of one of the $12$ dimensional maximum cases. Indeed, their upper left $12\times 12$ submatrix agrees with the matrix shown in \eqref{eq111}.
\begin{example}
The largest cardinality of a biangular line system in $\mathbb{R}^5$ with an irrational inner product is $20$. There are $12$ candidate Gram matrices, each corresponding to a single line system. The following candidate Gram matrix (with $\gamma:=-\alpha$, and $\delta:=-\beta$)
\begin{equation*}
C(\alpha,\beta,\gamma,\delta)=\left[\begin{smallmatrix}
 1 & \alpha & \alpha & \alpha & \alpha & \alpha & \delta & \beta & \beta & \beta & \beta & \beta & \alpha & \alpha & \alpha & \alpha & \beta & \beta & \beta & \beta \\
\alpha & 1 & \alpha & \beta & \beta & \alpha & \beta & \delta & \beta & \alpha & \alpha & \beta & \alpha & \alpha & \beta & \beta & \alpha & \alpha & \gamma & \delta \\
\alpha & \alpha & 1 & \alpha & \beta & \beta & \beta & \beta & \delta & \beta & \alpha & \alpha & \alpha & \alpha & \alpha & \delta & \alpha & \delta & \delta & \gamma \\
\alpha & \beta & \alpha & 1 & \alpha & \beta & \beta & \alpha & \beta & \delta & \beta & \alpha & \alpha & \beta & \alpha & \beta & \delta & \gamma & \alpha & \delta \\
\alpha & \beta & \beta & \alpha & 1 & \alpha & \beta & \alpha & \alpha & \beta & \delta & \beta & \alpha & \delta & \beta & \alpha & \gamma & \delta & \alpha & \alpha \\
\alpha & \alpha & \beta & \beta & \alpha & 1 & \beta & \beta & \alpha & \alpha & \beta & \delta & \alpha & \beta & \delta & \alpha & \delta & \alpha & \delta & \alpha \\
\delta & \beta & \beta & \beta & \beta & \beta & 1 & \alpha & \alpha & \alpha & \alpha & \alpha & \alpha & \gamma & \gamma & \gamma & \delta & \delta & \delta & \delta \\
\beta & \delta & \beta & \alpha & \alpha & \beta & \alpha & 1 & \alpha & \beta & \beta & \alpha & \alpha & \gamma & \delta & \delta & \gamma & \gamma & \alpha & \beta \\
\beta & \beta & \delta & \beta & \alpha & \alpha & \alpha & \alpha & 1 & \alpha & \beta & \beta & \alpha & \gamma & \gamma & \beta & \gamma & \beta & \beta & \alpha \\
\beta & \alpha & \beta & \delta & \beta & \alpha & \alpha & \beta & \alpha & 1 & \alpha & \beta & \alpha & \delta & \gamma & \delta & \beta & \alpha & \gamma & \beta \\
\beta & \alpha & \alpha & \beta & \delta & \beta & \alpha & \beta & \beta & \alpha & 1 & \alpha & \alpha & \beta & \delta & \gamma & \alpha & \beta & \gamma & \gamma \\
\beta & \beta & \alpha & \alpha & \beta & \delta & \alpha & \alpha & \beta & \beta & \alpha & 1 & \alpha & \delta & \beta & \gamma & \beta & \gamma & \beta & \gamma \\
\alpha & \alpha & \alpha & \alpha & \alpha & \alpha & \alpha & \alpha & \alpha & \alpha & \alpha & \alpha & 1 & \delta & \delta & \beta & \delta & \delta & \beta & \beta \\
\alpha & \alpha & \alpha & \beta & \delta & \beta & \gamma & \gamma & \gamma & \delta & \beta & \delta & \delta & 1 & \alpha & \alpha & \alpha & \beta & \delta & \delta \\
\alpha & \beta & \alpha & \alpha & \beta & \delta & \gamma & \delta & \gamma & \gamma & \delta & \beta & \delta & \alpha & 1 & \alpha & \beta & \gamma & \alpha & \delta \\
\alpha & \beta & \delta & \beta & \alpha & \alpha & \gamma & \delta & \beta & \delta & \gamma & \gamma & \beta & \alpha & \alpha & 1 & \delta & \alpha & \beta & \alpha \\
\beta & \alpha & \alpha & \delta & \gamma & \delta & \delta & \gamma & \gamma & \beta & \alpha & \beta & \delta & \alpha & \beta & \delta & 1 & \beta & \gamma & \gamma \\
\beta & \alpha & \delta & \gamma & \delta & \alpha & \delta & \gamma & \beta & \alpha & \beta & \gamma & \delta & \beta & \gamma & \alpha & \beta & 1 & \gamma & \alpha \\
\beta & \gamma & \delta & \alpha & \alpha & \delta & \delta & \alpha & \beta & \gamma & \gamma & \beta & \beta & \delta & \alpha & \beta & \gamma & \gamma & 1 & \beta \\
\beta & \delta & \gamma & \delta & \alpha & \alpha & \delta & \beta & \alpha & \beta & \gamma & \gamma & \beta & \delta & \delta & \alpha & \gamma & \alpha & \beta & 1 \\
\end{smallmatrix}\right]
\end{equation*}
yields a Gram matrix $C(\alpha^\ast,\beta^\ast,-\alpha^\ast,-\beta^\ast)$.\qed
\end{example}
\begin{example}
The largest biangular line system in $\mathbb{R}^6$ with irrational inner products is a unique configuration of $24$ lines, corresponding to the candidate Gram matrix (where $\gamma:=-\alpha$, and $\delta:=-\beta$):
\begin{equation*}
C(\alpha,\beta,\gamma,\delta)=\left[\begin{smallmatrix}
1 & \alpha & \alpha & \alpha & \alpha & \alpha & \delta & \beta & \beta & \beta & \beta & \beta & \alpha & \alpha & \alpha & \alpha & \alpha & \alpha & \alpha & \alpha & \alpha & \alpha & \alpha & \alpha \\
\alpha & 1 & \alpha & \beta & \beta & \alpha & \beta & \delta & \beta & \alpha & \alpha & \beta & \alpha & \alpha & \alpha & \alpha & \alpha & \alpha & \beta & \beta & \beta & \beta & \delta & \delta \\
\alpha & \alpha & 1 & \alpha & \beta & \beta & \beta & \beta & \delta & \beta & \alpha & \alpha & \alpha & \alpha & \alpha & \alpha & \beta & \beta & \alpha & \alpha & \delta & \delta & \beta & \beta \\
\alpha & \beta & \alpha & 1 & \alpha & \beta & \beta & \alpha & \beta & \delta & \beta & \alpha & \alpha & \alpha & \beta & \beta & \delta & \delta & \alpha & \alpha & \beta & \beta & \alpha & \alpha \\
\alpha & \beta & \beta & \alpha & 1 & \alpha & \beta & \alpha & \alpha & \beta & \delta & \beta & \alpha & \alpha & \delta & \delta & \beta & \beta & \beta & \beta & \alpha & \alpha & \alpha & \alpha \\
\alpha & \alpha & \beta & \beta & \alpha & 1 & \beta & \beta & \alpha & \alpha & \beta & \delta & \alpha & \alpha & \beta & \beta & \alpha & \alpha & \delta & \delta & \alpha & \alpha & \beta & \beta \\
\delta & \beta & \beta & \beta & \beta & \beta & 1 & \alpha & \alpha & \alpha & \alpha & \alpha & \alpha & \alpha & \gamma & \gamma & \gamma & \gamma & \gamma & \gamma & \gamma & \gamma & \gamma & \gamma \\
\beta & \delta & \beta & \alpha & \alpha & \beta & \alpha & 1 & \alpha & \beta & \beta & \alpha & \alpha & \alpha & \gamma & \gamma & \gamma & \gamma & \delta & \delta & \delta & \delta & \beta & \beta \\
\beta & \beta & \delta & \beta & \alpha & \alpha & \alpha & \alpha & 1 & \alpha & \beta & \beta & \alpha & \alpha & \gamma & \gamma & \delta & \delta & \gamma & \gamma & \beta & \beta & \delta & \delta \\
\beta & \alpha & \beta & \delta & \beta & \alpha & \alpha & \beta & \alpha & 1 & \alpha & \beta & \alpha & \alpha & \delta & \delta & \beta & \beta & \gamma & \gamma & \delta & \delta & \gamma & \gamma \\
\beta & \alpha & \alpha & \beta & \delta & \beta & \alpha & \beta & \beta & \alpha & 1 & \alpha & \alpha & \alpha & \beta & \beta & \delta & \delta & \delta & \delta & \gamma & \gamma & \gamma & \gamma \\
\beta & \beta & \alpha & \alpha & \beta & \delta & \alpha & \alpha & \beta & \beta & \alpha & 1 & \alpha & \alpha & \delta & \delta & \gamma & \gamma & \beta & \beta & \gamma & \gamma & \delta & \delta \\
\alpha & \alpha & \alpha & \alpha & \alpha & \alpha & \alpha & \alpha & \alpha & \alpha & \alpha & \alpha & 1 & \alpha & \alpha & \beta & \alpha & \beta & \alpha & \beta & \alpha & \beta & \alpha & \beta \\
\alpha & \alpha & \alpha & \alpha & \alpha & \alpha & \alpha & \alpha & \alpha & \alpha & \alpha & \alpha & \alpha & 1 & \beta & \alpha & \beta & \alpha & \beta & \alpha & \beta & \alpha & \beta & \alpha \\
\alpha & \alpha & \alpha & \beta & \delta & \beta & \gamma & \gamma & \gamma & \delta & \beta & \delta & \alpha & \beta & 1 & \beta & \alpha & \delta & \alpha & \delta & \beta & \gamma & \beta & \gamma \\
\alpha & \alpha & \alpha & \beta & \delta & \beta & \gamma & \gamma & \gamma & \delta & \beta & \delta & \beta & \alpha & \beta & 1 & \delta & \alpha & \delta & \alpha & \gamma & \beta & \gamma & \beta \\
\alpha & \alpha & \beta & \delta & \beta & \alpha & \gamma & \gamma & \delta & \beta & \delta & \gamma & \alpha & \beta & \alpha & \delta & 1 & \beta & \beta & \gamma & \alpha & \delta & \beta & \gamma \\
\alpha & \alpha & \beta & \delta & \beta & \alpha & \gamma & \gamma & \delta & \beta & \delta & \gamma & \beta & \alpha & \delta & \alpha & \beta & 1 & \gamma & \beta & \delta & \alpha & \gamma & \beta \\
\alpha & \beta & \alpha & \alpha & \beta & \delta & \gamma & \delta & \gamma & \gamma & \delta & \beta & \alpha & \beta & \alpha & \delta & \beta & \gamma & 1 & \beta & \beta & \gamma & \alpha & \delta \\
\alpha & \beta & \alpha & \alpha & \beta & \delta & \gamma & \delta & \gamma & \gamma & \delta & \beta & \beta & \alpha & \delta & \alpha & \gamma & \beta & \beta & 1 & \gamma & \beta & \delta & \alpha \\
\alpha & \beta & \delta & \beta & \alpha & \alpha & \gamma & \delta & \beta & \delta & \gamma & \gamma & \alpha & \beta & \beta & \gamma & \alpha & \delta & \beta & \gamma & 1 & \beta & \alpha & \delta \\
\alpha & \beta & \delta & \beta & \alpha & \alpha & \gamma & \delta & \beta & \delta & \gamma & \gamma & \beta & \alpha & \gamma & \beta & \delta & \alpha & \gamma & \beta & \beta & 1 & \delta & \alpha \\
\alpha & \delta & \beta & \alpha & \alpha & \beta & \gamma & \beta & \delta & \gamma & \gamma & \delta & \alpha & \beta & \beta & \gamma & \beta & \gamma & \alpha & \delta & \alpha & \delta & 1 & \beta \\
\alpha & \delta & \beta & \alpha & \alpha & \beta & \gamma & \beta & \delta & \gamma & \gamma & \delta & \beta & \alpha & \gamma & \beta & \gamma & \beta & \delta & \alpha & \delta & \alpha & \beta & 1 \\
\end{smallmatrix}\right].
\end{equation*}
The matrix $C(\alpha^\ast,\beta^\ast,-\alpha^\ast,-\beta^\ast)$ is positive semidefinite of rank $6$.\qed
\end{example}
\begin{example}[The Schl\"afli graph and related structures]\label{EX3}
In $\mathbb{R}^6$, there is a well-known spherical $2$-distance set of cardinality $27$ with set of inner products $\{-1/2,1/4\}$, related to the adjacency matrix of the Schl\"afli graph \cite{L}, \cite{NM}. Let
\begin{equation*}
C(\alpha,\beta,\gamma,\delta)=\left[\begin{smallmatrix}
 1 & \alpha & \alpha & \alpha & \alpha & \beta & \beta & \beta & \beta & \gamma & \gamma & \gamma & \gamma & \gamma & \gamma & \delta & \delta & \delta & \delta & \delta & \delta & \delta & \delta & \delta & \delta & \delta & \delta \\
\alpha & 1 & \alpha & \beta & \beta & \alpha & \alpha & \beta & \beta & \gamma & \gamma & \gamma & \delta & \delta & \delta & \gamma & \gamma & \gamma & \delta & \delta & \delta & \delta & \delta & \delta & \delta & \delta & \delta \\
\alpha & \alpha & 1 & \beta & \beta & \beta & \beta & \alpha & \alpha & \gamma & \gamma & \gamma & \delta & \delta & \delta & \delta & \delta & \delta & \gamma & \gamma & \gamma & \delta & \delta & \delta & \delta & \delta & \delta \\
\alpha & \beta & \beta & 1 & \alpha & \alpha & \beta & \alpha & \beta & \delta & \delta & \delta & \gamma & \gamma & \gamma & \delta & \delta & \delta & \delta & \delta & \delta & \gamma & \gamma & \gamma & \delta & \delta & \delta \\
\alpha & \beta & \beta & \alpha & 1 & \beta & \alpha & \beta & \alpha & \delta & \delta & \delta & \gamma & \gamma & \gamma & \delta & \delta & \delta & \delta & \delta & \delta & \delta & \delta & \delta & \gamma & \gamma & \gamma \\
\beta & \alpha & \beta & \alpha & \beta & 1 & \alpha & \alpha & \beta & \delta & \delta & \delta & \delta & \delta & \delta & \gamma & \gamma & \gamma & \delta & \delta & \delta & \gamma & \gamma & \gamma & \delta & \delta & \delta \\
\beta & \alpha & \beta & \beta & \alpha & \alpha & 1 & \beta & \alpha & \delta & \delta & \delta & \delta & \delta & \delta & \gamma & \gamma & \gamma & \delta & \delta & \delta & \delta & \delta & \delta & \gamma & \gamma & \gamma \\
\beta & \beta & \alpha & \alpha & \beta & \alpha & \beta & 1 & \alpha & \delta & \delta & \delta & \delta & \delta & \delta & \delta & \delta & \delta & \gamma & \gamma & \gamma & \gamma & \gamma & \gamma & \delta & \delta & \delta \\
\beta & \beta & \alpha & \beta & \alpha & \beta & \alpha & \alpha & 1 & \delta & \delta & \delta & \delta & \delta & \delta & \delta & \delta & \delta & \gamma & \gamma & \gamma & \delta & \delta & \delta & \gamma & \gamma & \gamma \\
\gamma & \gamma & \gamma & \delta & \delta & \delta & \delta & \delta & \delta & 1 & \alpha & \alpha & \gamma & \delta & \delta & \gamma & \delta & \delta & \gamma & \delta & \delta & \alpha & \beta & \beta & \alpha & \beta & \beta \\
\gamma & \gamma & \gamma & \delta & \delta & \delta & \delta & \delta & \delta & \alpha & 1 & \alpha & \delta & \gamma & \delta & \delta & \gamma & \delta & \delta & \gamma & \delta & \beta & \alpha & \beta & \beta & \alpha & \beta \\
\gamma & \gamma & \gamma & \delta & \delta & \delta & \delta & \delta & \delta & \alpha & \alpha & 1 & \delta & \delta & \gamma & \delta & \delta & \gamma & \delta & \delta & \gamma & \beta & \beta & \alpha & \beta & \beta & \alpha \\
\gamma & \delta & \delta & \gamma & \gamma & \delta & \delta & \delta & \delta & \gamma & \delta & \delta & 1 & \alpha & \alpha & \alpha & \beta & \beta & \alpha & \beta & \beta & \gamma & \delta & \delta & \gamma & \delta & \delta \\
\gamma & \delta & \delta & \gamma & \gamma & \delta & \delta & \delta & \delta & \delta & \gamma & \delta & \alpha & 1 & \alpha & \beta & \alpha & \beta & \beta & \alpha & \beta & \delta & \gamma & \delta & \delta & \gamma & \delta \\
\gamma & \delta & \delta & \gamma & \gamma & \delta & \delta & \delta & \delta & \delta & \delta & \gamma & \alpha & \alpha & 1 & \beta & \beta & \alpha & \beta & \beta & \alpha & \delta & \delta & \gamma & \delta & \delta & \gamma \\
\delta & \gamma & \delta & \delta & \delta & \gamma & \gamma & \delta & \delta & \gamma & \delta & \delta & \alpha & \beta & \beta & 1 & \alpha & \alpha & \alpha & \beta & \beta & \gamma & \delta & \delta & \gamma & \delta & \delta \\
\delta & \gamma & \delta & \delta & \delta & \gamma & \gamma & \delta & \delta & \delta & \gamma & \delta & \beta & \alpha & \beta & \alpha & 1 & \alpha & \beta & \alpha & \beta & \delta & \gamma & \delta & \delta & \gamma & \delta \\
\delta & \gamma & \delta & \delta & \delta & \gamma & \gamma & \delta & \delta & \delta & \delta & \gamma & \beta & \beta & \alpha & \alpha & \alpha & 1 & \beta & \beta & \alpha & \delta & \delta & \gamma & \delta & \delta & \gamma \\
\delta & \delta & \gamma & \delta & \delta & \delta & \delta & \gamma & \gamma & \gamma & \delta & \delta & \alpha & \beta & \beta & \alpha & \beta & \beta & 1 & \alpha & \alpha & \gamma & \delta & \delta & \gamma & \delta & \delta \\
\delta & \delta & \gamma & \delta & \delta & \delta & \delta & \gamma & \gamma & \delta & \gamma & \delta & \beta & \alpha & \beta & \beta & \alpha & \beta & \alpha & 1 & \alpha & \delta & \gamma & \delta & \delta & \gamma & \delta \\
\delta & \delta & \gamma & \delta & \delta & \delta & \delta & \gamma & \gamma & \delta & \delta & \gamma & \beta & \beta & \alpha & \beta & \beta & \alpha & \alpha & \alpha & 1 & \delta & \delta & \gamma & \delta & \delta & \gamma \\
\delta & \delta & \delta & \gamma & \delta & \gamma & \delta & \gamma & \delta & \alpha & \beta & \beta & \gamma & \delta & \delta & \gamma & \delta & \delta & \gamma & \delta & \delta & 1 & \alpha & \alpha & \alpha & \beta & \beta \\
\delta & \delta & \delta & \gamma & \delta & \gamma & \delta & \gamma & \delta & \beta & \alpha & \beta & \delta & \gamma & \delta & \delta & \gamma & \delta & \delta & \gamma & \delta & \alpha & 1 & \alpha & \beta & \alpha & \beta \\
\delta & \delta & \delta & \gamma & \delta & \gamma & \delta & \gamma & \delta & \beta & \beta & \alpha & \delta & \delta & \gamma & \delta & \delta & \gamma & \delta & \delta & \gamma & \alpha & \alpha & 1 & \beta & \beta & \alpha \\
\delta & \delta & \delta & \delta & \gamma & \delta & \gamma & \delta & \gamma & \alpha & \beta & \beta & \gamma & \delta & \delta & \gamma & \delta & \delta & \gamma & \delta & \delta & \alpha & \beta & \beta & 1 & \alpha & \alpha \\
\delta & \delta & \delta & \delta & \gamma & \delta & \gamma & \delta & \gamma & \beta & \alpha & \beta & \delta & \gamma & \delta & \delta & \gamma & \delta & \delta & \gamma & \delta & \beta & \alpha & \beta & \alpha & 1 & \alpha \\
\delta & \delta & \delta & \delta & \gamma & \delta & \gamma & \delta & \gamma & \beta & \beta & \alpha & \delta & \delta & \gamma & \delta & \delta & \gamma & \delta & \delta & \gamma & \beta & \beta & \alpha & \alpha & \alpha & 1 \\
\end{smallmatrix}\right].
\end{equation*}
Then $C(1,0,0,1)-I_{27}$ is the graph adjacency matrix of the Schl\"afli graph, and $C(1/4,-1/2,-1/2,1/4)$ is a spherical two-distance set spanning biangular lines. Application of Proposition~\ref{PINF} yields an infinite family of $27$ biangular lines in $\mathbb{R}^7$. It turns out that $C(1/4$, $-1/2,1/2,-1/4)$ is an additional, nonisometric example in $\mathbb{R}^6$, coming from a $4$-class association scheme \cite{hanaki}. The antipodal double of this set is a new example of spherical $5$-designs \cite{BAN9}. \qed
\end{example}

\begin{thebibliography}{11}
\bibitem{cocoa}
\textsc{J. Abbott, A.M. Bigatti}: CoCoALib: a C++ library for doing Computations in Commutative Algebra. Available at \url{http://cocoa.dima.unige.it/cocoalib}, ver.~0.99560 (2019)
\bibitem{BACHOC}
\textsc{C. Bachoc, F. Vallentin}: Optimality and uniqueness of the $(4,10,1/6)$ spherical code, {\it J. Combin. Theory Ser. A}, {\bf 116}, 195--204 (2009)
\bibitem{SUD}
\textsc{I. Balla, F. Dr\"axler, P. Keevash, B. Sudakov}: Equiangular lines and spherical codes in Euclidean space, {\it Invent. math.}, {\bf 211}, 179--212 (2018)
\bibitem{BAN9}
\textsc{E. Bannai, E. Bannai}: A survey on spherical designs and algebraic combinatorics on the sphere, {\it European J. Combin}., {\bf 30}, 1392--1425 (2009)
\bibitem{BBS}
\textsc{E. Bannai, E. Bannai, D. Stanton}: An upper bound for the cardinality of an $s$-distance subset in real Euclidean space, II, {\it Combinatorica}, {\bf 3}, 147--152 (1983)
\bibitem{BECK}
\textsc{T. Becker, V. Weispfenning}: Gr\"obner Bases, Springer--Verlag, New York (1993)
\bibitem{Best}
\textsc{D. Best}: Biangular vectors, MSc thesis, University of Lethbridge (2013)
\bibitem{KHA}
\textsc{D. Best, H. Kharaghani, H. Ramp}: Mutually unbiased weighing matrices, {\it Des. Codes Cryptogr.}, {\bf 76}, 237--256 (2015)
\bibitem{BLO}
\textsc{A. Blokhuis}: Few-Distance Sets, CWI Tract 7, CWI, Amsterdam (1984)
\bibitem{BOY}
\textsc{P. Boyvalenkov, K. Delchev}: On maximal antipodal spherical codes with few distances, {\it Elec. Notes Discr. Math.}, {\bf 57}, 85--90 (2017)
\bibitem{BRINK}
\textsc{G. Brinkmann}: Fast generation of cubic graphs, {\it J. Graph Theory}, {\bf 23} 139--149 (1996)
\bibitem{AEB}
\textsc{A.E. Brouwer}: Block designs. In: R. Graham, M. Gr\"otschel, L. Lov\'asz (eds.) Handbook of Combinatorics, vol. I, pp. 693--745. Elsevier, Amsterdam (1995)
\bibitem{IEEECODES}
\textsc{A. E. Brouwer, J. B. Shearer, N. J. A. Sloane, W. D. Smith}: A new table of constant weight codes, {\it IEEE Trans. Inform. Theory} {\bf 36}, 1334--1380 (1990)
\bibitem{COHN}
\textsc{H. Cohn, Y. Jiao, A. Kumar, S. Torquato}: Rigidity of spherical codes, {\it Geom. Topol.}, {\bf 15}, 2235--2273 (2011)
\bibitem{SPLAG}
\textsc{J. Conway, N.J.A. Sloane}: Sphere Packings, Lattices and Groups (third edition). Springer-Verlag New York (1999)
\bibitem{DGS}
\textsc{P. Delsarte, J.M. Goethals, J.J. Seidel}: Bounds on systems of lines and Jacobi polynomials, {\it Philips Res. Repts}, {\bf 30}, 91--105 (1975)
\bibitem{DGS2}
\textsc{P. Delsarte, J.M. Goethals, J.J. Seidel}: Spherical codes and designs, {\it Geometriae Dedicata} {\bf 6}, 363--388 (1977)
\bibitem{EZBOOK}
\textsc{T. Ericson, V. Zinoviev}: Codes on Euclidean Spheres, Elsevier Science (2001)
\bibitem{STEINER}
\textsc{M. Fickus, D.G. Mixon, J.C. Tremain}: Steiner equiangular tight frames, {\it Linear Algebra Appl.}, {\bf 436}, 1014--1027 (2012)
\bibitem{GKM}
\textsc{G. Greaves, J.H. Koolen, A. Munemasa, F. Sz\"oll\H{o}si}: Equiangular lines in Euclidean spaces, {\it J. Combin. Theory Ser. A}, {\bf 138}, 208--235 (2016)
\bibitem{BTF}
\textsc{J.I. Haas, J. Cahill, J. Tremain, P.G. Casazza}: Constructions of biangular tight frames and their relationships with equiangular tight frames, {\it preprint}, arXiv:1703.01786 [math.FA] (2017)
\bibitem{hanaki}
\textsc{A. Hanaki, I. Miyamoto}: Classification of association schemes of small order, {\it Discrete Math.}, {\bf 264}, 75--80 (2003)
\bibitem{SLOANE2}
\textsc{R.H. Hardin, N.J.A. Sloane}: McLaren's improved snub cube and other new spherical designs in three dimensions, {\it Discrete Comput. Geom.}, {\bf 15}, 429--441 (1996)
\bibitem{NJH}
\textsc{N.J. Higham}: Cholesky Factorization, {\it WIREs Comp. Stat.}, {\bf 1}, 251--254 (2009)
\bibitem{HOLZNEW}
\textsc{W.H. Holzmann, H. Kharaghani, S. Suda}: Mutually unbiased biangular vectors and association schemes. In: \textsc{C.J. Colbourn} (ed.) Algebraic Design Theory and Hadamard Matrices, 149--157 (2015)
\bibitem{KREDOCK}
\textsc{W.M. Kantor}: Codes, Quadratic Forms and Finite Geometries, {\it Proc. Sympos. Appl. Math.}, {\bf 50}, 153--177 (1995)
\bibitem{KO}
\textsc{P. Kaski, P.R.J. \"Osterg\aa rd}: Classification Algorithm for Codes and Designs, Springer, Berlin (2006)
\bibitem{BIPLANEPAT}
\textsc{P. Kaski, P.R.J. \"Osterg\aa rd}: There are exactly five biplanes with $k=11$, {\it J. Combin. Des.}, {\bf 16}, 117--127 (2008)
\bibitem{LS}
\textsc{P.W.H. Lemmens, J.J. Seidel}: Equiangular lines, {\it J. Algebra}, {\bf 27}, 494--512 (1973)
\bibitem{vLS}
\textsc{J.H. van Lint, J.J. Seidel}: Equilateral point sets in elliptic geometry, {\it Indag. Math.} {\bf 28}, 335--348 (1966)
\bibitem{L}
\textsc{P. Lison\v{e}k}: New maximal two-distance sets, {\it J. Combin. Theory Ser. A}, {\bf 77}, 318--338 (1997)
\bibitem{MCK}
\textsc{B.D. McKay, A. Piperno}: Practical graph isomorphism, II, {\it J. Symbolic Comput.}, {\bf 60}, 94--112 (2013)
\bibitem{DUSTIN}
\textsc{D.G. Mixon, H. Parshall}: The optimal packing of eight points in the real projective plane, {\it Exp. Math.}, to appear \url{https://doi.org/10.1080/10586458.2019.1641767} (2019)
\bibitem{NM}
\textsc{O.R. Musin, H. Nozaki}: Bounds on three- and higher-distance sets, {\it European J. Combin.}, {\bf 32}, 1182--1190 (2011)
\bibitem{NOZ}
\textsc{H. Nozaki}: A generalization of Larman--Rogers--Seidel's theorem, {\it Discrete Math.}, {\bf 311}, 792--799 (2011)
\bibitem{SHINO}
\textsc{H. Nozaki, M. Shinohara}: Maximal $2$-distance sets containing the regular simplex, preprint arXiv:1904.11351 [math.CO] (2019)
\bibitem{NS}
\textsc{H. Nozaki, S. Suda}: Weighing matrices and spherical codes, {\it J. Algebraic Combin.}, {\bf 42}, 283--291 (2015).
\bibitem{REA}
\textsc{R.C. Read}: Every one a winner, or how to avoid isomorphism search when cataloguing
combinatorial configurations, {\it Ann. Discrete Math.}, {\bf 2}, 107--120 (1978)
\bibitem{SZO}
\textsc{F. Sz\"oll\H{o}si, P.R.J. \"Osterg\aa rd}: Constructions of maximum few-distance sets in Euclidean spaces, {\it preprint}, arXiv:1804.06040 [math.MG] (2018)
\bibitem{SHAYNE}
\textsc{S.F.D. Waldron}: An Introduction to Finite Tight Frames, Birkh\"auser (2018)
\end{thebibliography}
\end{document}